\documentclass[11pt]{article}

\usepackage[paperwidth=8.5in,paperheight=11in,portrait,top=1in,bottom=1.in,left=1.15in,right=1.15in]{geometry}
\usepackage[utf8]{inputenc}
\usepackage{authblk}
\usepackage{blkarray}

\usepackage{amsfonts,amsmath,amssymb,amsthm}
\usepackage{latexsym,mathrsfs,mathtools,bm}
\usepackage{graphicx,subcaption,epsfig,caption,float,xcolor}
\usepackage{enumitem}
\usepackage{chngcntr}

\usepackage{tikz}
\usetikzlibrary{calc}

\usepackage[hidelinks]{hyperref}
\usepackage{bookmark}

\theoremstyle{plain}
\newtheorem{thm}{Theorem}[section]

\newtheorem{lem}[thm]{Lemma}
\newtheorem{prop}[thm]{Proposition}

\newtheorem{defi}[thm]{Definition}
\newtheorem{rem}[thm]{Remark}
\newtheorem{exam}[thm]{Example}

\numberwithin{equation}{section}

\newcommand{\cD}{\mathcal{D}}
\newcommand{\cI}{\mathcal{I}}

\newcommand{\cZ}{\mathcal{Z}}

\newcommand{\CC}{\mathbb{C}}
\newcommand{\II}{\mathbb{I}}
\newcommand{\JJ}{\mathbb{J}}
\newcommand{\NN}{\mathbb{N}}

\newcommand{\bv}{\mathrm{v}}
\newcommand{\bx}{\mathbf{x}}
\newcommand{\by}{\mathbf{y}}
\newcommand{\bz}{\mathbf{z}}

\numberwithin{equation}{section}

\title{\bf Bivariate $P$-polynomial association schemes }

\renewcommand*{\Affilfont}{\normalsize\small}
\author[1]{Pierre-Antoine Bernard}
\author[2]{Nicolas Cramp\'e}
\author[3]{Lo\"ic Poulain d'Andecy}
\author[4]{Luc Vinet}
\author[5]{Meri Zaimi\vspace{.5em}}

\affil[1,4,5]{Centre de Recherches Math\'ematiques, Universit\'e de Montr\'eal, \newline\vspace{.9em}
P.O. Box 6128, Centre-ville Station, Montr\'eal (Qu\'ebec), H3C 3J7, Canada.}

\affil[2]{Institut Denis-Poisson CNRS/UMR 7013 - Universit\'e de Tours - Universit\'e d'Orl\'eans, \newline\vspace{.9em}
Parc de Grandmont, 37200 Tours, France.}

\affil[3]{Laboratoire de math\'ematiques de Reims UMR 9008, Universit\'e de Reims Champagne-Ardenne,  \newline\vspace{.9em} 
Moulin de la Housse BP 1039, 51100 Reims, France.}

\affil[3]{IRL-CRM, CNRS UMI 3457, Universit\'e de Montr\'eal. \vspace{.9em} }
\affil[4]{IVADO, Montr\'eal (Qu\'ebec), H2S 3H1, Canada. \vspace{.9em}}

{
	\makeatletter
	\renewcommand\AB@affilsepx{: \protect\Affilfont}
	\makeatother
	\affil[ ]{E-mail addresses}
	\makeatletter
	\renewcommand\AB@affilsepx{, \protect\Affilfont}
	\makeatother
	\affil[1]{pierre-antoine.bernard@umontreal.ca}
	\affil[2]{crampe1977@gmail.com}
	\affil[3]{loic.poulain-dandecy@univ-reims.fr}
	\affil[4]{vinet@crm.umontreal.ca}\affil[5]{meri.zaimi@umontreal.ca}
}

\begin{document}

\date{\today} 
\maketitle
\vspace{15mm}
\begin{center}
\textbf{Abstract}\vspace{5mm}\\
\begin{minipage}{12cm}
Bivariate $P$-polynomial association scheme of type $(\alpha,\beta)$ are defined as a generalization of the $P$-polynomial association schemes. 
This generalization is shown to be equivalent to a set of conditions on the intersection parameters. 
A number of known higher rank association schemes are seen to belong to this broad class. Bivariate $Q$-polynomial association schemes are similarly defined.
\end{minipage}
\end{center}

\vspace{15mm}

\section{Introduction}

This paper is devoted to the generalization of the notion of $P$-polynomial association scheme to the case where the monovariate polynomials appearing in the definition of the latter
are replaced by bivariate polynomials. Numerous examples of bivariate $P$-polynomial association schemes are provided.

Let us recall the usual definitions. The set $\cZ=\{A_0,\dots,A_N\}$ is a symmetric association scheme with $N$ classes if 
the matrices $A_i$, called adjacency matrices, are non-zero $\bv\times \bv$ matrices with $0$ and $1$ entries
satisfying:
\begin{itemize}
 \item[(i)] $A_0=\II$ where $\II$ is the $\bv\times \bv$ identity matrix;
 \item[(ii)] $\displaystyle \sum_{i=0}^N A_i=\JJ$  where $\JJ$ is the $\bv\times \bv$ matrix filled with $1$;
 \item[(iii)] $A_i^t=A_i$ for $i=0,1,\dots N$ and $.^t$ stands for the transpose;
 \item[(iv)] The following relations hold
 \begin{equation}
  A_iA_j=A_jA_i=\sum_{k=0}^N p_{ij}^k A_k,
 \end{equation}
 where $p_{ij}^k$ are constants called intersection numbers.
\end{itemize}
This notion plays a role in various contexts. 
It appears in the theory of experimental design for the analysis of variance \cite{BM,Bai} and arises in the context of algebraic combinatorics and, in particular, in combinatorial designs and coding theory \cite{BI,God2}.
It also generalizes the character theory of representations of groups \cite{Bai,Zie2}. Indeed the matrices $A_i$ 
of an association scheme generate a commutative algebra, called Bose--Mesner algebra, which is related to the notion of character algebra.

Association schemes are very general structures far from being completely understood and classified. 
However, for a subclass of association schemes called $P$-polynomial, many connections with other topics allow a deeper understanding. 
For instance, matrices of a $P$-polynomial association scheme are in correspondence with distance matrices of distance-regular graphs.
They also satisfy by definition,
\begin{equation}
 A_i=v_i(A_1)\quad \text{for } i=0,1,\dots, N,
\end{equation}
where $v_i$ are polynomials of degree $i$ known to verify a three-term recurrence relation. As such, they give by Favard's theorem a set of orthogonal polynomials.
Imposing further that the association scheme is $Q$-polynomial, these polynomials $v_i$ must belong to the Askey scheme \cite{Leo,BI}. 

There exist many multivariate generalizations of the polynomials of the Askey scheme (see \cite{Gri,Tra,GI,GI2,HR,CFR,GW}). 
Some of these polynomials appear already in the context of association schemes, 
in the expression of the eigenvalues of the adjacency matrices \cite{Del,MS,MT,Bie,TAG,Kur}.
The goal of this paper consists in generalizing the notion of $P$-polynomial association schemes to a larger subclass of association schemes 
such that these multivariate polynomials appear naturally. We focus in this paper on the case of bivariate polynomials even if we believe that the case of multivariate 
polynomials can be treated similarly. In Section \ref{sec:def}, the bivariate $P$-polynomial association scheme is defined. 
More precisely, we give the definition of a bivariate $P$-polynomial association scheme of type $(\alpha,\beta)$.
The notion of type $(\alpha,\beta)$ corresponds to a feature of the bivariate polynomials which is also defined in Section \ref{sec:def}.
Then, in Section \ref{sec:met}, we define the notion of $(\alpha,\beta)$-metric association scheme imposing constraints on the intersection numbers and show 
that this notion is equivalent to be a bivariate $P$-polynomial association scheme of type $(\alpha,\beta)$.
This implies that the bivariate polynomials satisfy certain recurrence relations.
Section \ref{sec:idem} recalls the construction of the idempotents.
In Section \ref{sec:ex}, different examples are treated in detail. We show that the direct product of association schemes, 
the symmetrization of association schemes, the $24$-cell, the non-binary Johnson scheme and association schemes based on isotropic or attenuated spaces 
are bivariate $P$-polynomial association schemes. In Section \ref{sec:Qpol}, the definition of bivariate $Q$-polynomial association scheme is provided.
The symmetrization of association schemes is shown to also provide bivariate $Q$-polynomial association schemes as well. Section \ref{sec:outlooks} concludes this paper with some perspectives.

\section{Bivariate $P$-polynomial association schemes}

\subsection{Definition of bivariate $P$-polynomial association schemes \label{sec:def}}

This section provides the main definitions and, in particular, the definition of a bivariate $P$-polynomial association scheme of type $(\alpha,\beta)$.
Firstly, the notion of degree for a bivariate polynomial is introduced. 
In the following, we use the total order \textit{deg-lex} on the monomials denoted $\leq$ and defined by 
\begin{equation}
 x^my^n \leq x^iy^j \Leftrightarrow \begin{cases}
                                         m+n<i+j \\
                                         \text{or}\\
                                        m+n=i+j \quad\text{and}\quad n \leq j\,.
                                        \end{cases} \label{eq:ord}
\end{equation}
The degree, associated to the total order \textit{deg-lex}, of a polynomial $v(x,y)$ in two variables $x$ and $y$ is the couple $(i,j)$ such that $x^iy^j$ is the greatest monomial in $v(x,y)$.

Secondly, the polynomials playing an important role in the following have more structure. 
Let us introduce the partial order on monomials
\begin{equation}
 x^my^n \preceq_{(\alpha,\beta)}  x^i y^j \Leftrightarrow \begin{cases}
                                        m+\alpha n\leq i+\alpha j \\
                                         \text{and}\\
                                     \beta m+ n\leq \beta i+ j\,,
                                        \end{cases} \label{eq:partord}
\end{equation}
where  $0 \leq \alpha\leq 1$ and $0 \leq \beta < 1$. 
The previous constraints on the parameters $\alpha$ and $\beta$ have been chosen such that if $x^my^n \preceq_{(\alpha,\beta)}  x^i y^j$ then  $x^my^n \leq x^iy^j$. 
Note that we have obviously
\begin{equation}\label{monomial-order}
 x^my^n \preceq_{(\alpha,\beta)}  x^i y^j\ \ \ \Leftrightarrow\ \ \ x^{m+1}y^n \preceq_{(\alpha,\beta)}  x^{i+1} y^j\ \ \Leftrightarrow\ \ x^my^{n+1} \preceq_{(\alpha,\beta)}  x^i y^{j+1}\ .
 \end{equation}
We shall use the same symbol $\preceq_{(\alpha,\beta)}$ to also order pairs in $\mathbb{N}^2$. Examples of subsets of points $(m,n)$ smaller than $(i,j)$ are displayed in Figure \ref{fig:dom} for different values of the parameters $\alpha$ and $\beta$.
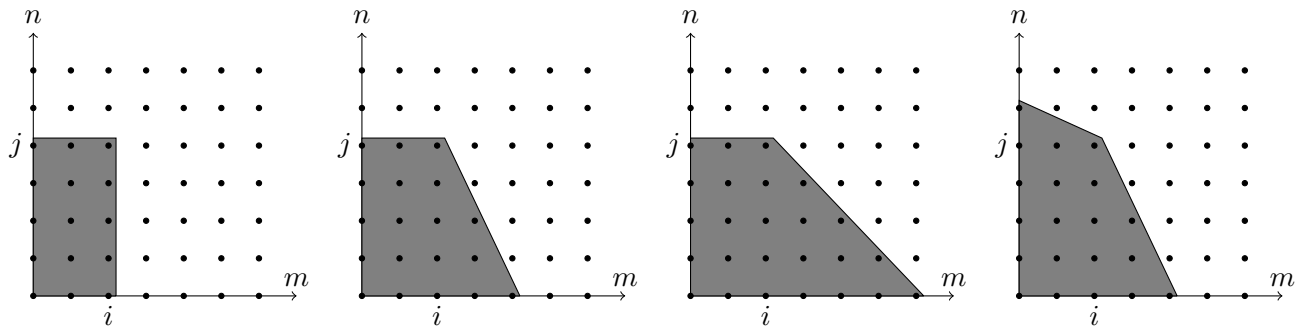
\begin{figure}[htbp]
\begin{center}
\begin{tikzpicture}[scale=0.5]
\draw[->] (0,0)--(7,0);\draw[->] (0,0)--(0,7);
\draw[fill=gray,opacity=0.1] (2.2,0)--(2.2,4.2) --(0,4.2) -- (0,0) -- (2.2,0);
\foreach \s in {0,1,...,6} {
\foreach \t in {0,1,...,6} {\draw [fill] (\s,\t-0) circle (0.07);}};
\draw (0,7) node[above] {$n$};
\draw (7,0) node[above] {$m$};
\draw (2,0) node[below] {$i$};
\draw (0,4) node[left] {$j$};
\end{tikzpicture}
\begin{tikzpicture}[scale=0.5]
\draw[->] (0,0)--(7,0);\draw[->] (0,0)--(0,7);
\draw[fill=gray,opacity=0.1] (4.2,0)--(2.2,4.2) --(0,4.2) -- (0,0) -- (4.2,0);
\foreach \s in {0,1,...,6} {
\foreach \t in {0,1,...,6} {\draw [fill] (\s,\t) circle (0.07);}};
\draw (0,7) node[above] {$n$};
\draw (7,0) node[above] {$m$};
\draw (2,0) node[below] {$i$};
\draw (0,4) node[left] {$j$};
\end{tikzpicture}
\begin{tikzpicture}[scale=0.5]
\draw[->] (0,0)--(7,0);\draw[->] (0,0)--(0,7);
\draw[fill=gray,opacity=0.1] (6.2,0)--(2.2,4.2) --(0,4.2) -- (0,0) -- (6.2,0);
\foreach \s in {0,1,...,6} {
\foreach \t in {0,1,...,6} {\draw [fill] (\s,\t) circle (0.07);}};
\draw (0,7) node[above] {$n$};
\draw (7,0) node[above] {$m$};
\draw (2,0) node[below] {$i$};
\draw (0,4) node[left] {$j$};
\end{tikzpicture}
\begin{tikzpicture}[scale=0.5]
\draw[->] (0,0)--(7,0);\draw[->] (0,0)--(0,7);
\draw[fill=gray,opacity=0.1] (4.2,0)--(2.2,4.2) --(0,5.2) -- (0,0) -- (4.2,0);
\foreach \s in {0,1,...,6} {
\foreach \t in {0,1,...,6} {\draw [fill] (\s,\t) circle (0.07);}};
\draw (0,7) node[above] {$n$};
\draw (7,0) node[above] {$m$};
\draw (2,0) node[below] {$i$};
\draw (0,4) node[left] {$j$};
\end{tikzpicture}
\end{center}
\begin{minipage}[t]{.24\linewidth}
\centering
\subcaption{$(\alpha,\beta)=(0,0)$}\label{fig:dom1} 
\end{minipage}
\begin{minipage}[t]{.24\linewidth}
\centering
\subcaption{$(\alpha,\beta)= (\frac{1}{2},0)$}\label{fig:dom2} 
\end{minipage}
\begin{minipage}[t]{.24\linewidth}
\centering
\subcaption{$(\alpha,\beta)= (1,0)$}\label{fig:dom3} 
\end{minipage}
\begin{minipage}[t]{.24\linewidth}
\centering
\subcaption{$(\alpha,\beta)= (\frac{1}{2},\frac{1}{2})$}\label{fig:dom4} 
\end{minipage}
\caption{
The points in the gray zone correspond to couple of integers $(m,n)$ smaller than $(i,j)$ for $\preceq_{(\alpha,\beta)}$ and for different values of $\alpha$ and $\beta$. \label{fig:dom}  }
\end{figure}

Thirdly, this leads to the following two definitions for bivariate polynomials and subsets of $\NN^2$.
\begin{defi}\label{def:compa}
A bivariate polynomial $v(x,y)$ is called $(\alpha,\beta)$-compatible of degree $(i,j)$ if the monomial $x^iy^j$ appears and all other monomials $x^my^n$ appearing are smaller than $x^iy^j$ for the order $\preceq_{(\alpha,\beta)}$.
\end{defi}
\begin{defi}\label{def:domain}
A subset $\cD$ of $\NN^2$ is called $(\alpha,\beta)$-compatible if for any $(i,j) \in \cD$, one gets 
\begin{equation} 
\Big( x^my^n\preceq_{(\alpha,\beta)} x^iy^j \Big) \Rightarrow  \Big( (m,n)\in \cD \Big).
\end{equation}
\end{defi}
\noindent In words, Definition \ref{def:domain} means that for $(i,j)\in \cD$, all $(m,n)$ such that $(m,n) \preceq_{(\alpha,\beta)} (i,j) $ are also in $\cD$ if the subset $\cD$ is $(\alpha,\beta)$-compatible, i.e. that $\cD$ is a downset of $(\mathbb{N}^2, \preceq_{(\alpha,\beta)})$.

Finally, we are in position to generalize the notion of $P$-polynomial association scheme.
\begin{defi} \label{def:bi}
Let $\cD \subset \NN^2$, $0 \leq \alpha\leq 1$, $0 \leq \beta<1$ and $\preceq_{(\alpha,\beta)}$ be the order \eqref{eq:partord}.
 The association scheme $\cZ$ is called bivariate $P$-polynomial of type $(\alpha,\beta)$ on the domain $\cD$ if these two conditions are satisfied:
 \begin{itemize}
  \item[(i)] there exists a relabeling of the adjacency matrices:
 \begin{equation}
  \{A_0,A_1,\dots, A_N\} = \{ A_{mn} \ |\ (m,n) \in \cD \},
 \end{equation}
such that, for $(i,j) \in \cD$,
\begin{equation}
 A_{ij}=v_{ij}(A_{10},A_{01})\,, \label{eq:vij}
\end{equation}
where  $v_{ij}(x,y)$ is a $(\alpha,\beta)$-compatible bivariate polynomial of degree $(i,j)$;
\item[(ii)] $\cD$ is $(\alpha,\beta)$-compatible.
 \end{itemize}
\end{defi}
Let us remark that the previous definition can also be given for other choices of the orders $\leq$ and $\preceq_{(\alpha,\beta)}$. However, all the examples 
we found (see Section \ref{sec:ex}) are in agreement with the definition given here. 
Let us also remark that the choice of $(\alpha,\beta)$ is not unique. In the following, we always choose $\alpha$ and $\beta$ as the smallest possible parameters. Finally, note that for simplicity we will sometimes omit mentioning explicitly the domain $\cD$ when discussing bivariate $P$-polynomial association schemes.

There are direct consequences of the previous definition:
\begin{itemize}
 \item The cardinality of $\cD$ is equal to $N+1$;
 \item $A_{00}=\II$, $\displaystyle \sum_{(i,j)\in\cD} A_{ij}=\JJ$;
 \item $A_{10}$ and $A_{01}$ generate the Bose--Mesner algebra;
 \item $v_{00}(x,y)=1$, $v_{10}(x,y)=x$ and $v_{01}(x,y)=y$;
 \item all monomials $ x^m y^n$ with non-zero coefficient appearing in $v_{ij}(x,y)$ are such that $(m,n)\in \cD$.
\end{itemize}
The bivariate polynomials $v_{ij}$ appearing in Definition \ref{def:bi} satisfy properties that are summarized in the following proposition and lemma.
\begin{prop}
	Let $\cZ$ be a bivariate $P$-polynomial association scheme of type $(\alpha,\beta)$ on the domain $\cD \subset \mathbb{N}^2$. Then, for all $(i,j)\in \cD$, the polynomial $v_{ij}(x,y)$ satisfying equation \eqref{eq:vij} is unique.
\end{prop}
\begin{proof}
	From the consequences listed above, for all $(i,j)\in \cD$ the polynomial $v_{ij}(x,y)$ of equation \eqref{eq:vij} is a linear combination of the monomials $x^my^n$ with $(m,n)\in \cD$. 
	Therefore, the Bose--Mesner algebra of the association scheme $\cZ$ is linearly generated by the matrices $A_{10}^mA_{01}^n$ with $(m,n) \in \cD$. 
	Since the cardinality of $\cD$ is equal to the dimension of the Bose--Mesner algebra, this generating set is linearly independent. 
	Suppose now that there is another $(\alpha,\beta)$-compatible bivariate polynomial of degree $(i,j)$ $v'_{ij}(x,y)\neq v_{ij}(x,y)$ such that $A_{ij}=v'_{ij}(A_{10},A_{01})$. 
	Since the monomials $x^my^n$ are linearly independent, this implies that there is a linear relation between the matrices $A_{10}^mA_{01}^n$ for $(m,n) \in \cD$, which contradicts their linear independence.    
\end{proof}
\begin{lem}\label{lem:xv}
 Let $v_{ij}(x,y)$ be the bivariate polynomials associated to a bivariate $P$-association scheme  of type $(\alpha,\beta)$ on $\cD \subset \NN^2$.
 For $(i,j)\in \cD$, there exist constants $\mu_{ij}^{mn}$ and $\nu_{ij}^{mn}$ such that
 \begin{eqnarray}
  x v_{i-1,j}(x,y) = \sum_{ (m,n)\preceq_{(\alpha,\beta)} (i,j) }\mu_{ij}^{mn}\  v_{mn}(x,y), \quad (i\geq 1)\label{eq:xvs}\\
   y v_{i,j-1}(x,y) = \sum_{ (m,n)\preceq_{(\alpha,\beta)} (i,j)}\nu_{ij}^{mn} \ v_{mn}(x,y), \quad (j\geq 1).\label{eq:yvs}
 \end{eqnarray}
\end{lem}
\proof Let $(i,j)\in \cD$. Any couple $(m,n)$ such that $(m,n)\preceq_{(\alpha,\beta)} (i,j)$ is also in $\cD$ by $(ii)$ of Definition \ref{def:bi}. 
Then, by the fact that $v_{mn}(x,y)$ are $(\alpha,\beta)$-compatible, the only monomials appearing in all the polynomials $v_{mn}(x,y)$ for $(m,n)\preceq_{(\alpha,\beta)} (i,j)$ are $x^my^n$ 
for $(m,n)\preceq_{(\alpha,\beta)} (i,j)$. Therefore, one gets
\begin{equation}
 \text{span}( v_{mn}(x,y)\ | \ (m,n)\preceq_{(\alpha,\beta)} (i,j)  ) = \text{span}( x^my^n\ | \ (m,n)\preceq_{(\alpha,\beta)} (i,j) ). \label{eq:sspan}
\end{equation}
Remarking that all the monomials $x^my^n$ present in $x v_{i-1,j}(x,y)$ satisfy the condition $(m,n)\preceq_{(\alpha,\beta)} (i,j)$ (see relation \eqref{monomial-order}) 
and using \eqref{eq:sspan}, one gets the equality \eqref{eq:xvs}.
Relation \eqref{eq:yvs} is proven similarly.
\endproof

\subsection{$(\alpha,\beta)$-metric association scheme \label{sec:met}}

It is well-known that for an association scheme the $P$-polynomial property is equivalent to the metric one, \textit{i.e.} that the intersection numbers satisfy the following constraints:
\begin{itemize}
 \item $p_{1j}^{j+1}\neq 0$ and $p_{1j}^{j-1}\neq 0$,
 \item $\left( p_{ij}^k\neq 0\right) \Rightarrow \left(|i-j|\leq  k \leq i+j\right)$.
\end{itemize}
For a bivariate $P$-polynomial association scheme, the intersection numbers read as follows
\begin{equation}
 A_{ij}A_{k\ell}= \sum_{(m,n)\in \cD}  p_{ij,k\ell}^{mn}\, A_{mn}.  \label{eq:intern}
\end{equation}
This subsection aims to generalize the  metric notion to bivariate $P$-polynomial association schemes.

If the polynomials $v_{ij}$ are the bivariate polynomials associated to a bivariate $P$-polynomial association scheme of type $(\alpha,\beta)$,
the intersection numbers corresponding to this association scheme are constrained as explained in the following proposition.
\begin{prop}\label{pr:z1} Let $\cZ=\{ A_{ij} \ |\ (i,j) \in \cD \}$ be an association scheme. The statements $(i)$ and $(ii)$ are equivalent:
\begin{itemize}
\item[(i)] $\cZ$ is a bivariate $P$-polynomial association scheme of type $(\alpha,\beta)$ on $\cD$;
  \item[(ii)] $\cD$ is $(\alpha,\beta)$-compatible and the intersection numbers satisfy, for  $(i,j),(i+1,j) \in \cD$,
  \begin{eqnarray}
   && p_{10,ij}^{i+ 1 , j}\neq 0,\ \  p_{10,i+1j}^{i , j}\neq 0, \label{eq:cm1} \\
   &&p_{10,ij}^{mn}\neq 0 \quad\left( \text{or}\ \ p_{10,mn}^{ij}\neq 0 \right)\quad \Rightarrow \quad   (m,n)\preceq_{(\alpha,\beta)} (i+1,j),\label{eq:cm3}
  \end{eqnarray}
  and, for  $(i,j),(i,j+1) \in \cD$,
    \begin{eqnarray}
   && p_{01,ij}^{i, j+1}\neq 0,\ \ p_{01,ij+1}^{i , j}\neq 0,  \label{eq:cm2}\\
 && p_{01,ij}^{mn}\neq 0 \quad\left(\text{or}\ \ p_{01,mn}^{ij}\neq 0\right) \quad \Rightarrow \quad   (m,n)\preceq_{(\alpha,\beta)} (i,j+1). \label{eq:cm4}
  \end{eqnarray}
\end{itemize}
\end{prop}

\proof $(i)\Rightarrow (ii)$:  From Lemma \ref{lem:xv}, relation \eqref{eq:xvs} holds and replacing $x$ and $y$ by $A_{10}$ and $A_{01}$ in it, one gets, for $(i,j)\in \cD$:
\begin{equation}
 A_{10} A_{i-1 j}= \sum_{ (m,n)\preceq_{(\alpha,\beta)} (i,j) }\alpha_{ij}^{mn}\ A_{mn}, \quad (i\geq 1).
\end{equation}
Comparing this equation with  \eqref{eq:intern} and knowing that the matrices $A_{ij}$ are independent, the following constraints on $p_{ij,k\ell}^{mn}$ hold:
\begin{equation}
 p_{10,i-1j}^{ij}\neq 0\quad \text{and} \quad \left( p_{10,i-1j}^{mn}\neq 0\right)\quad \Rightarrow \quad  \left( (m,n)\preceq_{(\alpha,\beta)} (i,j)\right).
\end{equation}
Since the association scheme is symmetric, the intersection numbers satisfy the following symmetry property: $p_{10,ij}^{mn}=0\  \Leftrightarrow \  p_{10,mn}^{ij}=0$.
This leads to relations \eqref{eq:cm1} and \eqref{eq:cm3}. Relations \eqref{eq:cm2} and \eqref{eq:cm4} are proven similarly starting from relation \eqref{eq:yvs} of Lemma \ref{lem:xv}.
\\
 $(ii)\Rightarrow (i)$: We use induction on $\leq$ to check that $A_{ij}=v_{ij}(A_{10},A_{01})$ with $v_{ij}$ being $(\alpha,\beta)$-compatible of degree $(i,j)$.
It is immediate for $i+j=1$. Now assume that $i\geq 1$. Then we have, using (\ref{eq:cm3}),
\begin{equation}
 A_{10}A_{i-1,j}=p_{10,i-1j}^{ij}A_{ij}+\sum_{(m,n)\preceq_{(\alpha,\beta)} (i,j)}p_{10,i-1j}^{mn}A_{mn}\ .
\end{equation}
Condition (\ref{eq:cm1}) ensures that $A_{ij}$ appears with a non-zero coefficient, so that this relation can be used for expressing $A_{ij}$ in terms of $A_{10}A_{i-1,j}$ and 
$A_{mn}$ with $(m,n)\preceq_{(\alpha,\beta)} (i,j)$. Since $(m,n)\preceq_{(\alpha,\beta)} (i,j)$ implies $(m,n)\leq (i,j)$, we can use the induction hypothesis on those $A_{mn}$ 
and clearly also on $A_{i-1,j}$. So we have that $A_{ij}$ is expressed as a polynomial $v_{ij}(x,y)$ evaluated in $A_{10},A_{01}$. 
Since $x^{i-1}y^j$ appears with a non-zero coefficient in $v_{i-1,j}$, we have that  $x^{i}y^j$ appears with a non-zero coefficient in $v_{ij}(x,y)$. 
The fact that $v_{ij}(x,y)$ is indeed $(\alpha,\beta)$-compatible follows now easily from the transitivity of $\preceq_{(\alpha,\beta)}$ and the property (\ref{monomial-order}).
If $i=0$, we can then assume that $j\geq1$ and use the same argument starting from $A_{01}A_{i,j-1}$, using now conditions (\ref{eq:cm2})-(\ref{eq:cm4}).
\endproof
For different choices of $\alpha$ and $\beta$ corresponding to the ones shown in Figure \ref{fig:dom}, the domains 
where $p_{10,ij}^{mn}$ and $p_{01,ij}^{mn}$ may be non-zero are displayed in Figure \ref{fig:recu}.
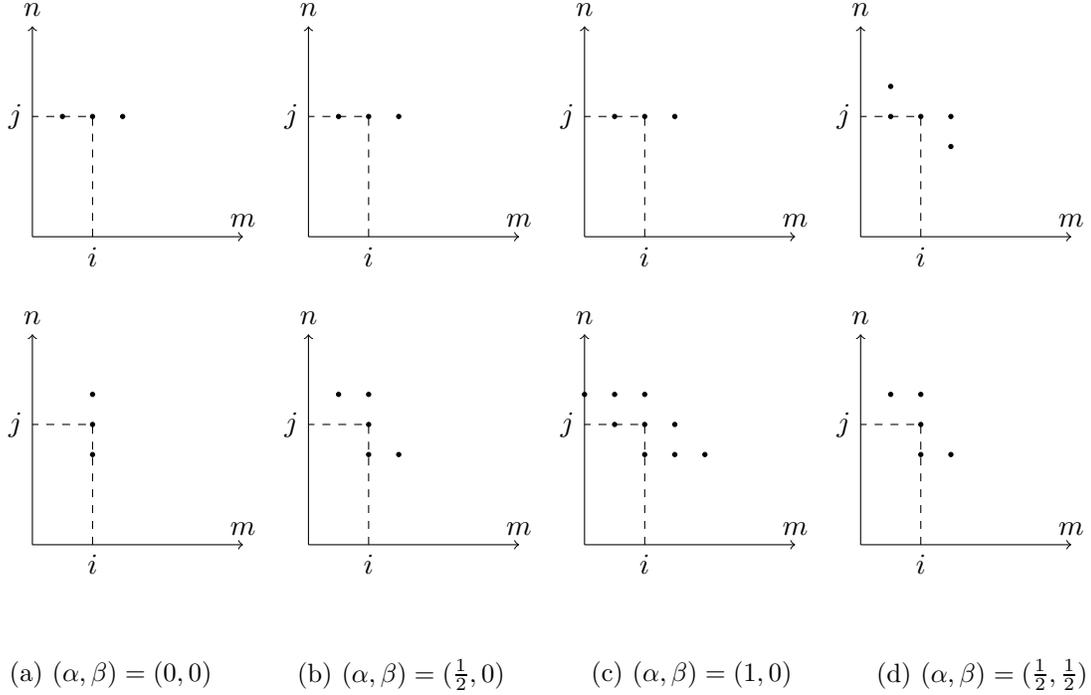
\begin{figure}[htbp]
\begin{center}
\begin{tikzpicture}[scale=0.4]
\draw[->] (0,0)--(7,0);\draw[->] (0,0)--(0,7);
\draw [fill] (2,4) circle (0.07);
\draw [fill] (3,4) circle (0.07);
\draw [fill] (1,4) circle (0.07);
\draw (0,7) node[above] {$n$};
\draw (7,0) node[above] {$m$};
\draw (2,0) node[below] {$i$};
\draw (0,4) node[left] {$j$};
\draw[dashed,thin] (2,0)--(2,4) --(0,4);
\end{tikzpicture}
\begin{tikzpicture}[scale=0.4]
\draw[->] (0,0)--(7,0);\draw[->] (0,0)--(0,7);
\draw [fill] (2,4) circle (0.07);
\draw [fill] (3,4) circle (0.07);
\draw [fill] (1,4) circle (0.07);
\draw (0,7) node[above] {$n$};
\draw (7,0) node[above] {$m$};
\draw (2,0) node[below] {$i$};
\draw (0,4) node[left] {$j$};
\draw[dashed,thin] (2,0)--(2,4) --(0,4);
\end{tikzpicture}
\begin{tikzpicture}[scale=0.4]
\draw[->] (0,0)--(7,0);\draw[->] (0,0)--(0,7);
\draw [fill] (2,4) circle (0.07);
\draw [fill] (3,4) circle (0.07);
\draw [fill] (1,4) circle (0.07);
\draw (0,7) node[above] {$n$};
\draw (7,0) node[above] {$m$};
\draw (2,0) node[below] {$i$};
\draw (0,4) node[left] {$j$};
\draw[dashed,thin] (2,0)--(2,4) --(0,4);
\end{tikzpicture}
\begin{tikzpicture}[scale=0.4]
\draw[->] (0,0)--(7,0);\draw[->] (0,0)--(0,7);
\draw [fill] (2,4) circle (0.07);
\draw [fill] (3,4) circle (0.07);
\draw [fill] (1,4) circle (0.07);
\draw [fill] (1,5) circle (0.07);
\draw [fill] (3,3) circle (0.07);
\draw (0,7) node[above] {$n$};
\draw (7,0) node[above] {$m$};
\draw (2,0) node[below] {$i$};
\draw (0,4) node[left] {$j$};
\draw[dashed,thin] (2,0)--(2,4) --(0,4);
\end{tikzpicture}

\end{center}

\begin{center}
\begin{tikzpicture}[scale=0.4]
\draw[->] (0,0)--(7,0);\draw[->] (0,0)--(0,7);
\draw [fill] (2,4) circle (0.07);
\draw [fill] (2,3) circle (0.07);
\draw [fill] (2,5) circle (0.07);
\draw (0,7) node[above] {$n$};
\draw (7,0) node[above] {$m$};
\draw (2,0) node[below] {$i$};
\draw (0,4) node[left] {$j$};
\draw[dashed,thin] (2,0)--(2,4) --(0,4);
\end{tikzpicture}
\begin{tikzpicture}[scale=0.4]
\draw[->] (0,0)--(7,0);\draw[->] (0,0)--(0,7);
\draw [fill] (2,4) circle (0.07);
\draw [fill] (2,3) circle (0.07);
\draw [fill] (2,5) circle (0.07);
\draw [fill] (3,3) circle (0.07);
\draw [fill] (1,5) circle (0.07);
\draw (0,7) node[above] {$n$};
\draw (7,0) node[above] {$m$};
\draw (2,0) node[below] {$i$};
\draw (0,4) node[left] {$j$};
\draw[dashed,thin] (2,0)--(2,4) --(0,4);
\end{tikzpicture}
\begin{tikzpicture}[scale=0.4]
\draw[->] (0,0)--(7,0);\draw[->] (0,0)--(0,7);
\draw [fill] (2,4) circle (0.07);
\draw [fill] (2,3) circle (0.07);

\draw [fill] (3,4) circle (0.07);

\draw [fill] (1,4) circle (0.07);

\draw [fill] (2,5) circle (0.07);
\draw [fill] (3,3) circle (0.07);
\draw [fill] (1,5) circle (0.07);
\draw [fill] (4,3) circle (0.07);
\draw [fill] (0,5) circle (0.07);
\draw (0,7) node[above] {$n$};
\draw (7,0) node[above] {$m$};
\draw (2,0) node[below] {$i$};
\draw (0,4) node[left] {$j$};
\draw[dashed,thin] (2,0)--(2,4) --(0,4);
\end{tikzpicture}
\begin{tikzpicture}[scale=0.4]
\draw[->] (0,0)--(7,0);\draw[->] (0,0)--(0,7);
\draw [fill] (2,4) circle (0.07);
\draw [fill] (2,3) circle (0.07);
\draw [fill] (2,5) circle (0.07);
\draw [fill] (3,3) circle (0.07);
\draw [fill] (1,5) circle (0.07);
\draw (0,7) node[above] {$n$};
\draw (7,0) node[above] {$m$};
\draw (2,0) node[below] {$i$};
\draw (0,4) node[left] {$j$};
\draw[dashed,thin] (2,0)--(2,4) --(0,4);
\end{tikzpicture}

\end{center}

\begin{center}
\begin{minipage}[t]{.21\linewidth}
\centering
\subcaption{$(\alpha,\beta)=(0,0)$}\label{fig:recu1} 
\end{minipage}
\begin{minipage}[t]{.21\linewidth}
\centering
\subcaption{$(\alpha,\beta)= (\frac{1}{2},0)$}\label{fig:recu2} 
\end{minipage}
\begin{minipage}[t]{.21\linewidth}
\centering
\subcaption{$(\alpha,\beta)= (1,0)$}\label{fig:recu3} 
\end{minipage}
\begin{minipage}[t]{.21\linewidth}
\centering
\subcaption{$(\alpha,\beta)= (\frac{1}{2},\frac{1}{2})$}\label{fig:recu4} 
\end{minipage}
\end{center}

\caption{The coordinate $(m,n)$ of the dots in the graphs at the top (resp. bottom) line represent when $p_{10,ij}^{mn}$ (resp. $p_{01,ij}^{mn}$) may be non-zero for different values of $(\alpha,\beta)$. \label{fig:recu}  }
\end{figure}
Note that by \eqref{eq:intern}, the constraints on the intersection numbers given in item $(ii)$ of Proposition \ref{pr:z1} can be equivalently viewed 
as constraints on the terms $A_{mn}$ with $(m,n) \in \cD$ appearing in the products $A_{10}A_{ij}$ and $A_{01}A_{ij}$. 
The expansions of these two products in terms of the matrices $A_{mn}$ correspond to the recurrence relations satisfied by the bivariate polynomials $v_{ij}$. 
One can read the type of these recurrence relations in Figure \ref{fig:recu}. Indeed, for example the case displayed in Figure \ref{fig:recu2} corresponds to the recurrence relations of the form:
\begin{eqnarray}
 x\,v_{ij}(x,y)&=&p_{10,ij}^{i+1j}\, v_{i+1j}(x,y)+p_{10,ij}^{ij}\, v_{ij}(x,y)+p_{10,ij}^{i-1j}\,v_{i-1j}(x,y),\\
 y\,v_{ij}(x,y)&=&p_{01,ij}^{ij+1}\,v_{ij+1}(x,y)+p_{01,ij}^{ij}\, v_{ij}(x,y)+p_{01,ij}^{ij-1}\,v_{ij-1}(x,y)\nonumber\\
&& +p_{01,ij}^{i-1j+1}\,v_{i-1j+1}(x,y)+p_{01,ij}^{i+1j-1}\, v_{i+1j-1}(x,y).
\end{eqnarray}

This proposition leads to the following definition.
\begin{defi}
The association scheme $\cZ=\{ A_{ij} \ |\ (i,j) \in \cD \}$ is called $(\alpha,\beta)$-metric inside the domain $\cD$ if $\cD$ is $(\alpha,\beta)$-compatible and if the associated intersection numbers $p_{ij,k\ell}^{mn}$ satisfy the conditions \eqref{eq:cm1}-\eqref{eq:cm4}.
\end{defi}
With the above definition, Proposition \ref{pr:z1} can be reformulated like this: 
an association scheme is $(\alpha,\beta)$-metric on $\cD$ if and only if it is bivariate $P$-polynomial of type $(\alpha,\beta)$ on $\cD$.

\subsection{Eigenvalues and idempotents \label{sec:idem}}

Let $\cZ=\{ A_{ij} \ |\ (i,j) \in \cD \}$ be a bivariate $P$-association scheme of  type $(\alpha,\beta)$.
Since the matrices $A_{ij}$ are pairwise commuting, they can be diagonalized in the same basis. 
The vector space $V$ of dimension $\bv$, on which the adjacency matrices act, can be decomposed as follows 
\begin{equation}
 V=\bigoplus_{\lambda \in \cD^\star} V_{\lambda}\,,
\end{equation}
where $\cD^\star$ is a set of labels with the same cardinality as $\cD$ (\textit{i.e.} $|\cD^\star|=|\cD|$) and $V_{\lambda}$ is a common eigenspace for all the matrices $A_{ij}$.
$\cD^\star$ denotes a choice of labelling for the common eigenspace and will mostly be a subset of $\mathbb{N}^2$ in the following sections. That $|\cD^\star|=|\cD|$ corresponds to the fact that the matrices $A_{ij}$ are linearly independent.  Since the sum of $A_{ij}$ is equal to $\JJ$, the common eigenspace containing the vector $(1,1,\dots,1)$ is of dimension 1. So we can take $\lambda_0 \in\cD^{\star}$ with $V_{\lambda_0}=\text{span}(1,1,\dots,1)$.

As usual in the context of association schemes, we denote by $E_{\lambda}$ with $\lambda \in \cD^\star$ the projector on the corresponding eigenspace: $E_{\lambda}V=V_{\lambda}$. They satisfy
\begin{eqnarray}
 &&E_{\lambda}E_{\lambda'}=\delta_{\lambda,\lambda'} E_{\lambda}\ , \qquad \sum_{\lambda \in \cD^\star} E_{\lambda}=\II \ , \qquad E_{\lambda_0}=\frac{1}{\bv} \JJ,\\
 &&  A_{ij} =\sum_{\lambda \in \cD^\star} p_{ij}(\lambda) E_{\lambda}\,, \label{eq:AE}
\end{eqnarray}
with $p_{ij}(\lambda) $ the eigenvalues of $A_{ij}$ in the subspace $V_{\lambda}$. The idempotents $E_{\lambda}$ also form a basis of the Bose--Mesner algebra.

With $\theta_{\lambda}=p_{10}(\lambda)$ and $\mu_{\lambda}=p_{01}(\lambda)$ the eigenvalues of $A_{10}$ and $A_{01}$, respectively, one gets
\begin{equation}
 A_{10} E_{\lambda}= \theta_{\lambda} E_{\lambda}\ , \quad A_{01} E_{\lambda}= \mu_{\lambda} E_{\lambda}\,.
\end{equation}
Since $A_{10}$ and $A_{01}$ generate the whole Bose--Mesner algebra, their eigenvalues characterize the eigenspaces  $V_{\lambda}$ \textit{i.e.} the couples 
$(\theta_{\lambda} ,\mu_{\lambda})$, $\lambda\in \cD^\star$, are different pairwise.

\begin{prop}\label{pr:z2}
 The eigenvalues $p_{ij}(\lambda) $ associated to a bivariate $P$-polynomial association scheme of type $(\alpha,\beta)$ satisfy
 \begin{equation}
  p_{ij}(\lambda)=v_{ij}(\theta_{\lambda},  \mu_{\lambda})\,, \label{eq:pv}
 \end{equation}
where $v_{ij}(x,y)$ is the bivariate polynomial of Definition \ref{def:bi}. 

Reciprocally, if an association scheme $\{A_{ij} \ | \  (i,j)\in \cD \} $ with $\cD$ $(\alpha,\beta)$-compatible has eigenvalues
satisfying \eqref{eq:pv} where $v_{ij}(x,y)$ is a bivariate polynomial $(\alpha,\beta)$-compatible of degree $(i,j)$, 
this scheme is a bivariate $P$-polynomial association scheme of type $(\alpha,\beta)$.
\end{prop}
\proof This result is a direct consequence of equation \eqref{eq:AE} and $A_{ij}=v_{ij}(A_{10},A_{01})$.\endproof

Relation \eqref{eq:AE} can be inverted and one gets
\begin{eqnarray}
E_{\lambda} =\frac{1}{\bv}\sum_{(i,j) \in \cD} q_{\lambda}(ij) A_{ij}\ . \label{eq:EA}
\end{eqnarray}
The parameters $q_{\lambda}(ij)$ are called dual eigenvalues.

\section{Examples of bivariate $P$-polynomial association schemes \label{sec:ex}}

It is obvious that all the association schemes with two classes $A_0,A_1,A_2$ are bivariate $P$-polynomial by setting $A_{10}=A_1$ and  $A_{01}=A_2$.
In the following subsections, a number of examples are given.

\subsection{Direct product of $P$-polynomial association schemes \label{sec:ex1}}

Let $A_0,\dots,A_D$ define an association scheme with intersection numbers $p_{ij}^k$ and let $\tilde{A}_0,\dots,\tilde{A}_{\widetilde{D}}$ provide another association scheme with intersection numbers $\widetilde{p}_{ij}^k$. 
The direct product is the association scheme defined by the Kronecker product of matrices:
\begin{equation}
 A_{ij}=A_i\otimes \tilde{A}_j\ ,\ \ \ \text{for } (i,j)\in \{0,\dots,D\}\times\{0,\dots,\widetilde{D}\} \,.
\end{equation}
Its intersection numbers are $p_{ij,kl}^{mn}=p_{ik}^m\, \widetilde{p}_{jl}^n$.

Assume that both association schemes are $P$-polynomial, so that we have:
\begin{equation} 
 A_i=v_i(A_1)\ \ \ \text{and}\ \ \ \tilde{A}_j=\tilde{v}_j(\tilde{A}_1)\ ,
\end{equation}
where $v_i$ (respectively, $\tilde{v}_j$) is a polynomial of degree $i$ (respectively, of degree $j$). We obtain immediately that the direct product is a  bivariate $P$-polynomial association scheme of type $(0,0)$, 
since we have:
\begin{equation}
 A_{ij}=v_{ij}(A_{10},A_{01})\,,\ \ \ \ \text{where }\  v_{ij}(x,y)=v_i(x)\tilde{v}_j(y). 
\end{equation}
The recurrence relations are given by
\begin{eqnarray}
 A_{10}A_{ij}=p_{1i}^{i-1}A_{i-1,j}+p_{1i}^{i}A_{ij}+p_{1i}^{i+1}A_{i+1,j}\,,\\
 A_{01}A_{ij}=\tilde{p}_{1j}^{j-1}A_{i,j-1}+\tilde{p}_{1j}^{j}A_{ij}+\tilde{p}_{1j}^{j+1}A_{i,j+1}\ .
\end{eqnarray}

\subsection{Symmetrization of association scheme with two classes \label{sec:ssP} }

In \cite{Del,MT}, the symmetrization of an association scheme has been defined and it has been shown that the expressions of the eigenvalues of the associated adjacency matrices 
are given by the multivariate Krawtchouk polynomials. 
We shall show that the symmetrization of an association scheme with two classes is a bivariate $P$-polynomial association scheme of type $(1/2,1/2)$. 
Note that any association scheme with two classes has the property of being (monovariate) $P$-polynomial and is equivalent to a strongly regular graph \cite{God}.

Let us recall the definition of the symmetrization. Let $A_0$, $A_1$ and $A_2$ define a $P$-polynomial association scheme whose associated matrices $L_i$, with entries 
$(L_i)_{hj}=p_{ij}^h$, read as follows (see \textit{e.g.} \cite{God})
\begin{equation}
 L_1=\begin{pmatrix}
      0& k & 0\\
      1& k-1-b & b \\
      0&c&k-c
     \end{pmatrix}\ ,\qquad L_2=\begin{pmatrix}
      0& 0 & \frac{b k}{c}\\
      0& b &  \frac{b k}{c}-b\\
     1&k-c&\frac{b k}{c}-1-k+c
     \end{pmatrix}.
\end{equation}
The matrices $A_i$ are $\bv \times \bv$ matrices with 
\begin{equation}
 \bv=\frac{k(b+c)+c}{c}.
\end{equation}
Let us define $A_{ij}$ by 
\begin{equation}
J(x)= (A_0 + x_1 A_1 +x_2A_2)^{\otimes N}= \sum_{\genfrac{}{}{0pt}{}{i,j=0}{i+j \leq N} }^N  A_{ij}\ x_1^i x_2^j, \label{eq:JN}
\end{equation}
where $x_1,x_2$ are abstract indeterminates.
The set $\{  A_{ij}    \ | \ i,j\geq 0,  i+j\leq N\}$, which is called symmetrization, defines an association scheme \cite{Del}.

An explicit form for the matrix $A_{ij}$ is the following:
\begin{equation}
	A_{ij}=\frac{1}{i!j!(N-i-j)!}\sum_{\pi\in S_N}\pi\cdot A_1^{\otimes i}\otimes A_2^{\otimes j}\otimes A_0^{\otimes N-i-j}\,, \label{eq:Aijsym}
\end{equation}
where the sum is over all the place permutations, and the prefactor ensures that each term appears only once. The sum over permutations $\pi$ in $S_N$ is the symmetrizer; it commutes with any $A_{ij}$ and, 
in particular, with $A_{10}$ and with $A_{01}$. A direct computation, with careful consideration of the prefactors, leads to
\begin{eqnarray}
 A_{10}A_{ij} & = &k(N-i-j+1)A_{i-1,j}+\bigl(i(k-1-b)+j(k-c)\bigr)A_{ij}+(i+1)A_{i+1,j}\notag\\
 && +c(j+1)A_{i-1,j+1}+b(i+1)A_{i+1,j-1}\,\label{eq:recs1},\\[0.5em]
 A_{01}A_{ij} & = & \displaystyle \frac{b k}{c}(N-i-j+1)A_{i,j-1}+\bigl(b i+j(\frac{b k}{c}-1-k+c)\bigr)A_{ij}+(j+1)A_{i,j+1}\notag\\
 && \displaystyle +(j+1)(k-c)A_{i-1,j+1}+(i+1)(\frac{b k}{c}-b)A_{i+1,j-1}\,.\label{eq:recs2}
\end{eqnarray}
From the previous results, we conclude that we have a bivariate $P$-polynomial association scheme of type $(1/2,1/2)$ (see Figure \ref{fig:dom4}) and
 $(1/2,1/2)$-compatible polynomials $v_{ij}$ of degree $(i,j)$ such that
\begin{equation}
 A_{ij}=v_{ij}(A_{10},A_{01})\,,\ \ \ \text{for any $i,j$ with $i+j\leq N$.}
\end{equation}

\begin{rem}
If $c=k$ then two terms cancel in the second recurrence relation. In this case, the association scheme is bivariate $P$-polynomial of type $(0,1/2)$. 
\end{rem}

\begin{exam} \label{ex:HO}
The scheme associated to the symmetrization of the Hamming scheme is called the ordered Hamming scheme \cite{Bie,MS}.
Let us recall that the Hamming scheme $H(2,q)$ corresponds to
\[k=2(q-1)\,,\ \ \ b=q-1\,,\ \ \ \ c=2\ .\] 
Then the recurrence relations of the ordered Hamming scheme become:
\begin{eqnarray}
 A_{10}A_{ij} & = &2(q-1)(N-i-j+1)A_{i-1,j}+(q-2)(i+2j)A_{ij}+(i+1)A_{i+1,j}\notag\\
 && +2(j+1)A_{i-1,j+1}+(q-1)(i+1)A_{i+1,j-1}\,,\\[0.5em]
 A_{01}A_{ij} & = & \displaystyle (q-1)^2(N-i-j+1)A_{i,j-1}+\bigl(i(q-1)+j(q-2)^2)\bigr)A_{ij}+(j+1)A_{i,j+1}\notag\\
 && \displaystyle +2(j+1)(q-2)A_{i-1,j+1}+(i+1)(q-1)(q-2)A_{i+1,j-1}\,.
\end{eqnarray}
For $q=2$, the second relation becomes a three-term recurrence relation (see the preceding remark).
\end{exam}

We now compute the eigenvalues of $A_{ij}$. For an association scheme with 2 classes, one gets \cite{God}
\begin{eqnarray}
 &&A_0= E_0 +  E_1  +  E_2\,,\\
 &&A_1=k E_0 + \theta E_1  + \tau E_2\,,\\
 &&A_2=(\bv-1-k)E_0-(\theta+1) E_1 -(\tau+1)E_2\,,
\end{eqnarray}
where the eigenvalues $\theta$ and $\tau$ are related to the parameters $b$ and $c$ as follows
\begin{equation}
 b=-(\theta+1)(\tau+1), \qquad c=k+\theta \tau,\qquad \bv=\frac{k(b+c)+c}{c}.
\end{equation}
Now, for $(i,j) \in \mathcal{D}$ one can consider the matrices $E_{ij}$ defined by
\begin{equation}
	E_{ij}=\frac{1}{i!j!(N-i-j)!}\sum_{\pi\in S_N}\pi\cdot E_1^{\otimes i}\otimes E_2^{\otimes j}\otimes E_0^{\otimes N-i-j}\,, \label{eq:Eijsym}
\end{equation}
and observe that they are the projectors onto the eigenspaces of $A_{10}$ and $A_{01}$ associated to eigenvalues $\theta_{ij}$ and $\mu_{ij}$:
\begin{equation}
A_{10} E_{ij} = \theta_{ij}E_{ij}, \quad A_{01} E_{ij} = \mu_{ij}E_{ij},
\end{equation}
where
\begin{equation}
 \theta_{ij}=(N-i-j)k+i\theta+j\tau, \qquad \mu_{ij}=(N-i-j)\frac{kb}{c}-i(\theta+1)-j(\tau+1).
\end{equation}
Let us take $\cD^* = \cD$ and rewrite the generating function $J(x)$ \eqref{eq:JN} using the idempotents $E_i$ 
\begin{equation}
 J(x)=(1+kx_1+\frac{kb}{c}x_2)^N \left(E_0 +\frac{1+\theta x_1-(\theta+1)x_2}{1+kx_1+\frac{kb}{c}x_2}E_1 +\frac{1+\tau x_1-(\tau+1)x_2}{1+kx_1+\frac{kb}{c}x_2}E_2 \right)^{\otimes N}.
\end{equation}
By remarking that $(E_0+X_1E_1+X_2E_2)^{\otimes N}=\sum_{i+j\leq N} X_1^i X_2^j E_{ij}$, one gets
\begin{equation}
 J(x)=\sum_{i+j\leq N}(1+kx_1+\frac{kb}{c}x_2)^{N-i-j} (1+\theta x_1-(\theta+1)x_2)^i (1+\tau x_1-(\tau+1)x_2)^j E_{ij}.
\end{equation}
Using relations \eqref{eq:AE} and \eqref{eq:pv}, one finds another expression for $J(x)$:
\begin{equation}
 J(x) =\sum_{i+j\leq N}\left( \sum_{m+n\leq N}v_{mn}(\theta_{ij},  \mu_{ij}) \  x_1^m x_2^n  \right)E_{ij}.
\end{equation}
Comparing both expressions of $J(x)$, the generating functions of the polynomials $v_{ij}$ are obtained:
\begin{eqnarray}
 \sum_{m+n\leq N}v_{mn}(\theta_{ij},  \mu_{ij}) \  x_1^m x_2^n= (1+kx_1+\frac{kb}{c}x_2)^{N-i-j} (1+\theta x_1-(\theta+1)x_2)^i (1+\tau x_1-(\tau+1)x_2)^j.\ \ 
\end{eqnarray}
The R.H.S. of this relation can be identified with the generating function of the bivariate Krawtchouk polynomials (see \textit{e.g.} \cite{GVZ}):
\begin{equation}
 v_{mn}(\theta_{ij},  \mu_{ij})= \sqrt{\frac{N!}{(N-m-n)!m!n!}}   \left(\sqrt{k}\right)^m \left(\sqrt{\frac{kb}{c}}\right)^n P_{mn}(i,j;N).
\end{equation}
The functions $P_{mn}(i,j;N)$ defined in \cite{GVZ} are proportional to the bivariate Krawtchouk polynomials \cite{HR,IT,Ili} and are defined from a matrix $R$ of $SO(3)$ which reads as follows in terms of the parameters used here
\begin{equation}
 R=\begin{pmatrix}
    \theta\sqrt{\frac{\tau+1}{(\tau-\theta)(k-\theta)}  } & \sqrt{\frac{(\theta+1)(\tau\theta+k)  }{ (\theta-\tau)(k-\theta)}} & \sqrt{\frac{k(\tau+1)}{(\tau-\theta)(k-\theta)}}\\
    -\tau\sqrt{\frac{\theta+1}{(\tau-\theta)(\tau-k)} } & \sqrt{\frac{ (\tau+1)(\tau\theta+k)}{(\tau-\theta)(k-\tau)}} & -\sqrt{\frac{k(\theta+1)}{(\tau-\theta)(\tau-k)}}\\
    \sqrt{\frac{k(\tau\theta+k)}{(k-\theta)(k-\tau)}}  & \sqrt{\frac{k(\theta+1)(\tau+1)}{(\theta-k)(k-\tau)}} & \sqrt{\frac{\tau\theta+k}{(k-\theta)(k-\tau)}}
   \end{pmatrix}.
\end{equation}
Let us remark that the recurrence relations \eqref{eq:recs1}-\eqref{eq:recs2} of $v_{mn}(x,y)$ can be recovered from the ones of $P_{mn}(i,j;N)$ given in \cite{GVZ}.

\subsection{A generalized $24-$cell}

In \cite{MMW} (see Theorem 3.6), an association scheme with 4 classes which generalizes the $24$-cell is studied. It is proven that it is $Q$-polynomial but not $P$-polynomial. We shall show that it is bivariate $P$-polynomial.
The matrices $L_i$, with entries $(L_i)_{kj}=p_{ij}^k$, are given by $L_0=\II_5$, 
\begin{eqnarray}
 &&L_1=\begin{pmatrix}
        0 & 16\ell s^2 & 0 & 0 & 0 \\
        1 & 2(\ell-1)s(4s+1) & (4s-1)(4s+1)& 2(\ell-1)s(4s-1)& 0\\
        0 & 8\ell s^2 & 0 & 8\ell s^2 & 0\\
         0 & 2(\ell-1)s(4s-1) & (4s-1)(4s+1)& 2(\ell-1)s(4s+1)& 1\\
         0 & 0 & 0 & 16\ell s^2 & 0
       \end{pmatrix},\\
        &&  L_2=\begin{pmatrix}
            0 & 0  & 2(4s-1)(4s+1) & 0 & 0\\
            0 & (4s-1)(4s+1) & 0 & (4s-1)(4s+1) & 0\\
            1 & 0 & 32s^2-4 & 0 & 1 \\
               0 & (4s-1)(4s+1) & 0 & (4s-1)(4s+1) & 0\\
               0 & 0  & 2(4s-1)(4s+1) & 0 & 0
           \end{pmatrix},\\
        && L_3=\begin{pmatrix}
        0 & 0 & 0 & 16\ell s^2 & 0 \\
        0 & 2(\ell-1)s(4s-1) & (4s-1)(4s+1)& 2(\ell-1)s(4s+1)& 1\\
        0 & 8\ell s^2 & 0 & 8\ell s^2 & 0\\
         1 & 2(\ell-1)s(4s+1) & (4s-1)(4s+1)& 2(\ell-1)s(4s-1)& 0\\
         0 & 16\ell s^2 & 0 & 0 &  0
       \end{pmatrix},
\end{eqnarray}
and
\begin{eqnarray}
       &&L_4=\begin{pmatrix}
              0 & 0 & 0 & 0 & 1\\
               0 & 0 & 0 & 1 & 0\\
                0 & 0 & 1 & 0 & 0\\
                 0 & 1 & 0 & 0 & 0\\
                  1 & 0 & 0 & 0 & 0
             \end{pmatrix}.
\end{eqnarray}
The fact that this scheme is not $P$-polynomial reads in the fact that $L_1$ is not tridiagonal.
Let us define 
\begin{equation}
 A_{00}=A_0\ , \quad A_{10}=A_2\ , \quad A_{01}=A_3\ , \quad A_{11}=A_1\ , \quad A_{20}=A_4\ .
\end{equation}
By direct computation (we recall that $A_i$ and $L_i$ satisfy the same relations), these matrices satisfy
\begin{eqnarray}
 && A_{11}=\frac{1}{(4s-1)(4s+1)} A_{01}A_{10}-A_{01}, \\
 && A_{20}=\frac{1}{2(4s-1)(4s+1)} A_{10}^2-\frac{2(8s^2-1)}{(4s-1)(4s+1)} A_{10}-A_{00}.
\end{eqnarray}
This demonstrates that it is a bivariate $P$-polynomial association scheme of type $(0,0)$ (see Figure \ref{fig:dom1}). 

\subsection{Non-binary Johnson association scheme}

The non-binary Johnson scheme is a generalization of the Johnson scheme which has eigenvalues that can be expressed in terms of bivariate polynomials 
formed of Krawtchouk and Hahn polynomials \cite{Dun,TAG}. We show here that it is a bivariate $P$-polynomial association scheme of type $(1,0)$.

We start by recalling the definition of the non-binary Johnson scheme, which can be found in \cite{TAG}. 
Let $K=\{0,1,2,\dots,r-1\}$, where $r$ is an integer greater than $1$ ($r \geq 2$), and consider the $n$-fold Cartesian product $K^n$, where $n$ is a positive integer. 
For a vector $\bx$ in $K^n$ with components $\bx_i$, the weight $w(\bx)$ is defined as the number of non-zero components of $\bx$, that is
\begin{equation}
	w(\bx) = \big|\{ i \ | \ \bx_i \neq 0  \}\big|. \label{eq:weight}
\end{equation}   
For two vectors $\bx,\by \in K^n$, the number of equal non-zero components $e(\bx,\by)$ and the number of common non-zeros $c(\bx,\by)$ are also defined:
\begin{align}
	e(\bx,\by) = \big|\{ i \ | \ \bx_i=\by_i \neq 0  \}\big|, \quad c(\bx,\by) = \big|\{ i \ | \ \bx_i\neq 0, \ \by_i\neq 0 \}\big|. \label{eq:ecnonzeros}
\end{align}
Consider a fixed weight number $k$. Note that we must have $0 \leq k \leq n$ by definition \eqref{eq:weight}. The set
\begin{equation}
	X = \{ \bx \in K^n \ | \ w(\bx)=k \} \,,
\end{equation} 
together with all the non-empty relations\footnote{The relations $R_{ij}$ of \cite{TAG} have been relabeled as follows: $i \mapsto i+j$ and $j \mapsto j$.}
\begin{equation}
	R_{ij} = \{ (\bx,\by) \in X^2 \ | \ e(\bx,\by)=k-i-j, \ c(\bx,\by)=k-j \}\,, \label{eq:RijNBJ}
\end{equation}    
define a symmetric association scheme called the non-binary Johnson scheme and  denoted $J_r(k,n)$, following the notation of \cite{TAG}. 
From this definition, one can construct the adjacency matrices $A_{ij}$ of the non-binary Johnson scheme. 
These are $|X|\times |X|$ matrices whose entries, labeled by the couples $(\bx,\by)\in X^2$, take the value one if $(\bx,\by) \in R_{ij}$ and zero otherwise.
In particular, for $i=j=0$, the adjacency matrix $A_{00}$ is the identity matrix since $(\bx,\by) \in R_{00}$ if and only if $\bx=\by$. 
More generally, one can observe from the definitions \eqref{eq:ecnonzeros} and \eqref{eq:RijNBJ} that two vectors $\bx,\by$ such that $(\bx,\by) \in R_{ij}$ have:
\begin{enumerate}
\setlength\itemsep{1pt}
	\item[(i)] $e(\bx,\by)=k-i-j$ equal non-zero components;
	\item[(ii)] $c(\bx,\by)-e(\bx,\by)=i$ unequal common non-zero components;
	\item[(iii)] $k-c(\bx,\by)=j$ uncommon non-zero components;
	\item[(iv)] $j$ uncommon zero components;
	\item[(v)] $n-k-j$ common zero components.
\end{enumerate}
It is therefore seen that the non-empty relations $R_{ij}$ are such that 
\begin{equation}
	0\leq i \leq k-j, \quad 0 \leq j \leq \min\{k,n-k\}.
\end{equation}

When $r=2$, it is seen that $J_2(k,n)$ is the Johnson scheme $J(n,k)$ (see \textit{e.g.} \cite{BI}). Indeed, in this case the alphabet $K$ is binary and we must have $e(\bx,\by)=c(\bx,\by)$, which implies that only the relations $R_{ij}$ with $i=0$ and $0 \leq j \leq \min\{k,n-k\}$ are non-empty. Moreover, the couples $(\bx,\by) \in R_{0j}$ are such that the Hamming distance of the vectors $\bx$ and $\by$ (\textit{i.e.}\ the number of unequal components) is the constant $2j$. In what follows, we will suppose $r\geq 3$.

Using the properties (i)--(v), it is possible to compute the following relations for the adjacency matrices:
\begin{align}
	A_{10}A_{ij} =& (k-i-j+1)(r-2) A_{i-1,j} + \left( i(r-3)+j(r-2) \right) A_{ij} +(i+1)A_{i+1,j}, \label{eq:rec10NBJ} \\
	A_{01}A_{ij} =& (k-i-j+1)(r-2)j A_{i-1,j} + (k-i-j+1)(n-k-j+1)(r-1)A_{i,j-1} \nonumber \\ 
	&+ (i+1)j A_{i+1,j} + (j+1)^2 A_{i,j+1} + (i+1)(n-k-j+1)(r-1) A_{i+1,j-1} \nonumber \\
	&+ (j+1)^2(r-2) A_{i-1,j+1} +j\left( k-i-j+(r-2)i+(n-k-j)(r-1)\right)A_{i,j}. \label{eq:rec01NBJ}
\end{align} 
As an example, the coefficient before the matrix $A_{i-1,j}$ in \eqref{eq:rec10NBJ} is computed as follows. 
For two vectors $\bx$ and $\by$ such that $(\bx,\by) \in R_{i-1,j}$, one needs to count the number of vectors $\bz$ such that $(\bx,\bz) \in R_{10}$ and $(\bz,\by) \in R_{ij}$. 
Using the properties (i)--(v), one finds that such vectors $\bz$ must have all components equal to those of $\bx$ except for one component $\bz_s$ at some coordinate $s$ 
which must be such that $\bz_s \neq 0$ and $\bz_s \neq \bx_s=\by_s$. There are $k-(i-1)-j$ possibilities for the coordinate $s$ because of (i), 
and $r-2$ possibilities for the value of $\bz_s \in K$. One then takes the product of these possibilities to obtain the total number of vectors $\bz$, 
which gives the coefficient written in \eqref{eq:rec10NBJ}. The other coefficients are found similarly. 

From the relations \eqref{eq:rec10NBJ} and \eqref{eq:rec01NBJ}, we can deduce using Proposition \ref{pr:z1} that the non-binary Johnson scheme is a bivariate $P$-polynomial association scheme of type $(1,0)$ (see Figure \ref{fig:recu3}).
Note that if $k\leq n-k$, then the domain $\cD \subseteq \mathbb{N}^2$ of the couples $(i,j)$ for the adjacency matrices $A_{ij}$ is the triangle $i+j \leq k$, whether if $n-k < k$, the domain is the same triangle 
truncated horizontally at $j=n-k$. In both cases, the domain $\cD$ is $(1,0)$-compatible (see Figure \ref{fig:dom3}).

In \cite{TAG}, the eigenvalues of the non-binary Johnson scheme $J_r(k,n)$ are labeled by couples of integers $(x,y)\in \cD^* = \cD $ and given explicitly in terms of 
bivariate polynomials\footnote{The following change of variables has been applied to the eigenvalues given in \cite{TAG} in order to fit with our conventions: $i \mapsto i+j, x\mapsto x+y, y\mapsto x$. 
The definitions of the polynomials have been also changed: $K_i(N,p,x)\mapsto K_i(x,N,p)$ and $E_i(N,p,x)\mapsto E_i(x,N,p)$.}:
\begin{equation}
	p_{ij}(x,y) = (r-1)^jK_{i}(x,k-j,r-1)E_{j}(y,n-x,k-x), \label{eq:pijNBJ}
\end{equation}
where $(i,j),(x,y) \in \cD$ and for $i=0,1,\dots,N$,
\begin{align}
	&K_i(x,N,p) = \sum_{\ell=0}^i(-1)^\ell (p-1)^{i-\ell} {x\choose \ell} {N-x\choose i-\ell}, \label{eq:krawtchouk} \\
	&E_i(x,N,p) = \sum_{\ell=0}^i (-1)^\ell {x\choose \ell} {p-x\choose i-\ell} {N-p-x\choose i-\ell}. \label{eq:eberlein}
\end{align}
Expression \eqref{eq:krawtchouk} is the Krawtchouk polynomial while \eqref{eq:eberlein} is the Eberlein polynomial. 
The latter can be expressed in terms of the (dual) Hahn polynomial and is known to provide the eigenvalues of Johnson scheme. 
As a byproduct of our approach in this paper, we have obtained recurrence relations for the polynomial \eqref{eq:pijNBJ}. Indeed, because of Proposition \ref{pr:z2}, we must have
\begin{equation}
	p_{ij}(x,y) = v_{ij}(p_{10}(x,y),p_{01}(x,y)),
\end{equation}
where $v_{ij}(x,y)$ is the (unique) bivariate polynomial such that $A_{ij}=v_{ij}(A_{10},A_{01})$ for the non-binary Johnson scheme (see Definition \ref{def:bi}). Therefore, the recurrence relations \eqref{eq:rec10NBJ} and \eqref{eq:rec01NBJ} for the adjacency matrices imply the same recurrence relations for the polynomials 
$p_{ij}(x,y)$ given in \eqref{eq:pijNBJ}, with the replacements $A_{ij}\mapsto p_{ij}(x,y)$.

\begin{rem}
	There is a connection between the non-binary Johnson scheme and the ordered Hamming scheme, 
	which we recall is the symmetrization of the Hamming scheme $H(2,q)$ (see example \ref{ex:HO} above). Indeed, consider the case $q=2$, and let $\{A_{ij} \ | \ i,j\geq 0, i+j\leq N \}$ be the set of adjacency matrices of the ordered Hamming scheme, as in Section \ref{sec:ssP}. These matrices act on a vector space $V$ of dimension $4^N$ with basis vectors $e_i$ for $i=1,\dots,4^N$. Consider now the set $W_k=\{e_i \ | \ e_1^t A_{N-k,0} \; e_i \neq 0 \}$ for $k$ any integer such that $0\leq k \leq N/2$, and denote by 
	$\overline{A}_{ij}$ the restriction of the matrices $A_{ij}$ on $W_k$. Then, using \eqref{eq:Aijsym}, one can show that the set of matrices $\{\overline{A}_{2s,j} \ | \ 0 \leq s \leq k, \  0 \leq j \leq N-k-s \}$ is a symmetric association scheme corresponding to the non-binary Johnson scheme $J_3(N,N-k)$ (up to a relabeling of indices). Put differently, $J_3(N,N-k)$ 
	can be viewed as a particular projection of the ordered Hamming scheme. This connection between these two bivariate $P$-polynomial association schemes is analogous to the embedding of the Johnson scheme $J(n,k)$ in the Hamming scheme $H(n,2)$ that was described algebraically in terms of projection matrices in \cite{BCV}.              
\end{rem}

\subsection{Association schemes based on isotropic spaces}

Association schemes based on dual polar spaces are well-known examples of $P$- and $Q$-polynomial schemes \cite{BCN, Stan}. 
They are obtained by considering vector spaces of dimension $D$ defined over finite fields $\mathbb{F}_q$ and equipped with a non-degenerate form $\mathfrak{B}$. 

For such scheme, the set of vertices $X$ (\textit{i.e.} the labels of the columns and rows of its adjacency matrices $A_i$) is composed of the \textit{maximal} isotropic subspaces of $\mathbb{F}_q^D$. 
These are the largest subspaces $V \subset \mathbb{F}_q^D$ such that the evaluation of the form $\mathfrak{B}(v_1,v_2)$  vanishes for any two vectors $v_1, v_2 \in V$.  
By Witt's theorem, they all have the same dimension $N \geq D/2$. The relations $\{R_i\}_{0 \leq i \leq N}$ giving the non-zero entries of the adjacency matrices $A_i$ are given by
\begin{equation}
 R_{i} = \{(V,V') \in  X \times X \ |\  \text{dim}(V \cap V') = N - i \}\,.
\end{equation}
A generalization was introduced in \cite{Rie} by dropping the \textit{maximality} condition on the subspaces. 
Indeed, it was shown that a set of vertices composed of the isotropic $d$-subspaces of $\mathbb{F}_q^D$ equipped with a non-degenerate form $\mathfrak{B}$ still gives a symmetric association scheme 
if the following relations are introduced:
\begin{equation}
 R_{ij} = \{(V,V') \in  X \times X \ |\  \text{dim}(V \cap V') = d - i - j, \ \text{dim}(V^\perp \cap V') = d - i \}\,,
\label{refiso}
\end{equation}
where $d\leq N$ and $V^\perp$ is the subspace composed of vectors $v_1 \in \mathbb{F}_q^D$ verifying $\mathfrak{B}(v_1,v_2) = 0$ for any vector $v_2 \in V$. Note that this scheme has for domain,
\begin{equation}
\mathcal{D} = \{(i,j)\ | \ 0 \leq i \leq d, \ 0 \leq j \leq N - d, \  0 \leq i+j \leq d\}\,.
\end{equation}
The spherical functions associated to the lattice of isotropic $d$-subpaces and thus the eigenvalues $p_{ij}(mn)$ of the scheme were computed in \cite{Stan}. This set of subspaces also has a combinatorial design interpretation as the $d$th fiber of the uniform poset consisting of all the isotropic subspaces \cite{ter90poset}. One naturally recovers the dual polar schemes for $d = N$, but the $P$- and $Q$-polynomial properties are lost in general for $d \neq N$. Our claim is that the bivariate $P$-polynomial property holds for all $d \in \{1,2,\dots \lfloor N/2\rfloor\}$.

While computing all the coefficients $p_{ij,k\ell}^{mn}$ of these schemes remains an open problem, general observations can be made regarding vanishing intersection parameters. 
For any two isotropic $d$-subspaces $V$ and $V'$ in relation $R_{01}$ or $R_{10}$, one has $\text{dim}(V \cap V') = d-1$ and thus
\begin{equation}
 V = V \cap V' \oplus \text{span}\{v_1\} \quad \text{and} \quad V' = V \cap V' \oplus \text{span}\{v_2\}\,,
\label{is1}
\end{equation}
where $v_1$,$v_2 \in \mathbb{F}_q^D$.  For any third isotropic $d$-subspaces $U$, it follows that
\begin{equation}
| \text{dim}(V\cap U) - \text{dim}(V'\cap U) | \leq 1, \quad \text{and}\quad | \text{dim}(V\cap U^\perp) - \text{dim}(V'\cap U^\perp) | \leq 1\,.
\end{equation}
These inequalities and the definition of the relations $R_{ij}$ then imply that
\begin{equation}
|i+j - k - \ell| \geq 2 \quad \text{or} \quad |i-k|\geq 2 \quad \Rightarrow \quad p_{10, k\ell}^{ij} = p_{01, k\ell}^{ij} = 0\,.
\end{equation}
Other vanishing intersection parameters can also be identified for the case $p_{10, k\ell}^{ij}$. 
Then, one is interested in the possible relations $R_{ij}$ between $U$, $V$ and $V'$ given that $(V,V') \in R_{10}$. 
In addition to  \eqref{is1}, we get that $\mathfrak{B}(v_1,v_2) \neq 0$ and thus
\begin{equation}
 \text{dim}(V\cap U) - \text{dim}(V'\cap U)  = 1 \quad \Rightarrow \quad \exists \ v_3 = v_1 + r \in V \cap U, \ \text{with}\  r \in V \cap V'\,.
\end{equation}
In particular, $v_3$ is in $V^\perp$ but not in  $V'^\perp$. To have $\text{dim}(V'^\perp \cap U) \geq \text{dim}(V^\perp \cap U)$, 
one would therefore require the existence of at least one vector $v'$ contained in $V'^\perp \cap U$ but not in $V^\perp \cap U$. 
By construction, it would verify $\mathfrak{B}(v', v_1) \neq 0$ and  $\mathfrak{B}(v', r) = 0$ such that $\mathfrak{B}(v', v_3) \neq 0$. 
Since $U$ has to be isotropic and $v_3 \in U$, such vector $v'$ cannot exist and one finds
\begin{equation}
 |i + j - k - \ell | =  1 \quad \Rightarrow \quad \ell = j\,.
\end{equation}
In terms of conditions on intersection parameters, it reads
\begin{equation}
  |i + j - k - \ell | =  1,\  \ell \neq j \quad \Rightarrow \quad p_{10, ij}^{k\ell} = 0\,.
\end{equation}

From these observations and the fact that $\mathcal{D}$ is $(1,1/2)$-compatible for $d \leq N/2$, we get from Proposition \ref{pr:z1} that the isotropic $d$-subspaces 
with the set of relations $R_{ij}$ yields a bivariate $P$-association schemes of type $(1,1/2)$. 
To illustrate this result, let us consider the case where $\mathfrak{B}$ is a symplectic form, $d = 2$ and $D \geq 6$. 
The domain $\mathcal{D}$ is then small enough to allow the direct computation of all intersection coefficients $p_{10, ij}^{k\ell}$ and $p_{10, ij}^{k\ell}$ explicitly through simple combinatorial arguments. 
The matrices $L_{10}$ and $L_{01}$, of entries $(L_{10})_{k\ell, ij} = p_{10,ij}^{k\ell}$ and $(L_{01})_{k\ell,ij} = p_{01,ij}^{k\ell}$, are given by

\[
L_{10} =
\begin{blockarray}{ccccccc}
00 & 10 & 01 & 11 & 20 & 02 \\
\begin{block}{(cccccc)c}
  0 & (q+1)q^{D - 3} & 0 & 0 &   0 & 0 & 00 \\
  1 & (q-1) q^{D -4} &  (q^{D - 4} - 1)&  0 & q^{D - 2} & 0 & 10\\
  0 & (q-1) q^{D -4} & q^{D-4} & q^{D-2} & 0 & 0 & 01 \\
  0 & 0 & q & (2q^2 - 1)q^{D- 5} & (q-1)q^{D-3} & q(q^{D-6} - 1) & 11 \\
  0 & (q+1) & 0 & (q+1)(q^{D-4} - 1) & (q^2 - 1)q^{D-4} & 0 & 20 \\
  0 & 0 & 0 & (q^2-1)(q+1)q^{D-5} & 0 & (q+1)q^{D-5} & 02 \\
\end{block}
\end{blockarray}
\]
 
\[
 L_{01} =
\begin{blockarray}{ccccccc}
00 & 10 & 01 & 11 & 20 & 02\\
\begin{block}{(cccccc)c}
  0 & 0 & \frac{(q+1)(q^{D-3} -q)}{(q-1)} & 0 & 0 & 0 & 00 \\
  0 & q^{D-4} - 1 & \frac{q^{D-4} - 1}{q-1} & q^2\frac{q^{D-4} - 1}{q-1} & 0 & 0 & 10 \\
  1 & q^{D - 4} & \frac{q^{D-4} - 1}{(q-1)} + q^2 - 2 & q^{D - 3} & 0 & q^3\frac{q^{D - 6} -1}{(q-1)} & 01 \\
  0 & q & 1 & \frac{(2q^2 - 1)(q^{D-5} -1)}{(q-1)} & q^{D-3} & \frac{q(q^{D-6} - 1)}{(q-1)} & 11 \\
  0 & 0 & 0 & \frac{(q+1)(q^{D-4} - 1)}{(q-1)} & (q+1)(q^{D-4} -1) & 0 & 20 \\
  0 & 0 & (q+1)^2 & (q+1)^2q^{D-5} & 0 & (q+1)\left(\frac{q^{D-5} -q^2 - q + 1}{q-1} \right) & 02 \\
\end{block}
\end{blockarray}
\]
Since the entries of these matrices give the coefficients in relations of the Bose--Mesner algebra, they allow to express $A_{10}$, $A_{01}$ and $A_{11}$ as the following polynomials of $A_{10}$ and $A_{01}$,
\begin{eqnarray}
A_{11}& =& q^{-1} A_{10} A_{01}- q^{D-5}  A_{01} - q^{-1} (q^{D-4} - 1)A_{10}\,,\\
A_{20} &=& \frac{1}{q+1} A_{10}^2  - \frac{(q-1)q^{D-4}}{(q+1)} A_{10} - \frac{(q-1)q^{D-4}}{(q+1)} A_{01}  - q^{D - 3} A_{00}\,,\\
A_{02} &=& \frac{1}{(q+1)^2}A_{01}^2 -\frac{1}{q(q+1)^2} A_{10} A_{01} + \left(\frac{q^{D-5}}{(q+1)^2}  - \frac{q^{D-4} - 1}{(q^2 - 1)(q+1)} - \frac{q^2 - 2}{(q+1)^2}\right)A_{01}\\ 
&+& \left( \frac{(q^{D-4} - 1)}{q(q+1)^2} - \frac{q^{D-4} - 1}{(q^2 - 1)(q+1)}\right) A_{10} - \frac{q^{D-3} - q}{q^2 - 1} A_{00}\,, \nonumber
\end{eqnarray}
which highlights the bivariate $P$-polynomial nature of the scheme.

\subsection{Association schemes based on attenuated spaces}

To explore another family of association schemes, let us consider a vector space of dimension $D + L$ over the finite field $\mathbb{F}_q$ 
and one of its subspaces $W$ of dimension $L$. For a subspace $V \subseteq \mathbb{F}_q^{D+ L}$, 
the quotient of $V+W$ by $W$ is denoted $(V+W)/W$. The set $X$ of subspaces $V$ of dimension $d$ with intersection $V \cap W = \emptyset$ is the $d$th fiber of a uniform poset and defines what is called an \textit{attenuated space} \cite{ter90poset, Wen16}.
Let us focus on the case $L\geq d$ (see Remark \ref{rem:Ld} below).
It was observed in \cite{WGL} that such a set equipped with relations $R_{ij}$ given by
\begin{equation}\label{R:as}
R_{ij} = \{(V,V') \in X \times X \ |\  \text{dim}(V\cap V' ) = d - i - j, \ \text{dim}( (V+W) /W \cap (V'+W)/W) = d - j \}\,,
\end{equation} 
defines a symmetric association scheme of domain
\begin{equation}
\mathcal{D} = \{(i,j) \in \mathbb{N}^2 \ | \ j \leq \text{min}(d, D-d), \ i+j \leq  d \}\,. 
\end{equation}
This domain $\mathcal{D}$ is $(1,0)$-compatible. 
A formula for all its intersection parameters $p_{ij, k\ell}^{mn}$ was also provided in \cite{WGL}. 
In the case of $p_{ij, k\ell}^{10}$ and $p_{ij, k\ell}^{01}$, the expressions simplify greatly and can be obtained through simple combinatorial arguments. 
This yields the following relations between the adjacency matrices,
\begin{equation}
 \begin{split}
        A_{01} A_{ij} &= q^{2j + i + L - 1} [d-i-j+1]_q[D-d-j+1]_qA_{i j-1} + [j+1]_q^2 q^{i} A_{i j+1} \\
        &  + [d-i-j+1]_q[j]_q(q^L - q^{i-1})q^{i + j }A_{i-1 j}+ [i+1]_q[j]_q q^{i + j + 1 }A_{i+1 j}  \\
        &+ [j+1]_q^2(q^L - q^{i-1})A_{i-1 j+1} + [i+1]_q[D-d-j+1]_q q^{2j + L - 1}A_{i+1 j-1} \\
       &+ [j]_q\big( [D-d-j]_q q^{L + 1 + j} + [d-i-j]_q q^{j + 2i + 1 } + [i]_q(q^L - q^{i-1})q^{j + 1 } + [j]_	q(q-1) q^{ L }\big)A_{ij}\,,
 \end{split}
  \label{asr1}
 \end{equation}
 and
 \begin{equation}
 \begin{split}
     A_{10}A_{ij} &= (q^L -q^{i-1}) [d-i-j+1]_q q^{i+j-1} A_{i-1j} + [i+1]_q q^{i+j} A_{i+1j} \\ &+ \big((q^L - 1)[i+j]_q - [i]_q q^{i+j-1} + (q-1)q^{i+j}[d-i-j]_q [i]_q\big) A_{ij}\,,
 \end{split}
 \label{asr2}
 \end{equation}
where $[n]_q = (q^n - 1)/(q-1)$ are $q$-numbers. From \eqref{asr1} and \eqref{asr2}, one can check that the scheme verifies conditions \eqref{eq:cm1}-\eqref{eq:cm4} with respect to the partial order $\preceq_{(1,0)}$. 
Since the domain $\mathcal{D}$ is also $(1,0)$-compatible, Proposition \ref{pr:z1} implies that this scheme is bivariate $P$-polynomial of type $(1,0)$. 

There thus exist polynomials $v_{ij}$ verifying $A_{ij} = v_{ij}(A_{10}, A_{01})$. These are solutions of  \eqref{asr1} and \eqref{asr2} interpreted as recurrence relations. 
Alternatively, they can be obtained using \eqref{eq:pv} and the spectrum $p_{ij}(m,n)$ of the adjacency matrices $A_{ij}$. An expression for the latter can be found in \cite{Dun2,Kur} and is given by, for $ (i,j) \in \cD$,
\begin{equation}
  p_{ij}(m,n) = q^{jL} K_i(d-j, L ; q ; n) Q_{j}(D-n,d-n;q;m)\,,
  \label{spAtt}
\end{equation}
where $(m,n) \in \cD^* =\cD$ labels the eigenspaces of the matrices $A_{ij}$ and $K_i$, $Q_{j}$ are respectively $q$-Krawtchouk polynomials and $q$-Hahn polynomials,
 \begin{equation}
     K_i(j, L ; q ; n) = \sum_{k = 0}^i (-1)^{i-k} q^{k L + \binom{i-k}{2}} \begin{bmatrix}
j-k \\
j-i
\end{bmatrix}_q \begin{bmatrix}
j-n \\
k
\end{bmatrix}_q,
 \end{equation}
 \begin{equation}
     Q_{i}(k,j;q;m) = \sum_{\ell = 0}^i (-1)^{i-\ell} q^{\ell m + \binom{i-\ell}{2}} \begin{bmatrix}
j-\ell \\
j-i
\end{bmatrix}_q
\begin{bmatrix}
j-m \\
\ell
\end{bmatrix}_q
\begin{bmatrix}
k-j+\ell-m \\
\ell
\end{bmatrix}_q\,,
\end{equation}
and $\left[\begin{smallmatrix} a \\ b \end{smallmatrix}\right]_q$ is the $q$-binomial coefficient. Therefore, the polynomials $v_{ij}$ are obtained from the relation
\begin{equation}
v_{ij}( \theta_{mn}, \mu_{mn}) =   q^{jL} K_i(d-j, L ; q ; n) Q_{j}(D-n,d-n;q;m)\,, 
\end{equation}
where
\begin{equation}
\mu_{mn} = q^L\big(q^{m}[d-m-n]_q[D-d+1-m]_q - [d-n]_q\big)\,,
\end{equation}
and
\begin{equation}
\theta_{mn} = -[d]_q + q^L [d-n]_q\,.
 \end{equation}

\begin{rem}
This scheme was constructed as a generalization of the Grassmann scheme $J_q(D,d)$ and bilinear scheme $H_q(D,L)$. 
The former is recovered by taking $L=0$ and the latter by imposing $d = D$. 
Similarities in the spectrum of the adjacency matrices and in the relations \eqref{eq:rec10NBJ}-\eqref{eq:rec01NBJ} and \eqref{asr1}-\eqref{asr2} also suggest 
that schemes based on attenuated spaces offer $q$-deformations of non-binary Johnson schemes $J_2(D,d)$. 
\end{rem}

 \begin{rem}\label{rem:Ld}  
 For $L< d$, relations $R_{ij}$ given by \eqref{R:as} still define an association scheme with adjacency matrices verifying \eqref{asr1}-\eqref{asr2} and with a spectrum given by \eqref{spAtt}. However the domain becomes 
 \begin{equation}
\mathcal{D} = \{(i,j) \in \mathbb{N}^2 \ | \ j \leq \text{min}(d, D-d), \ i \leq \text{min}(d-j, L) \}\,. 
\end{equation}
This domain is not $(1,0)$-compatible. Therefore in this case, it is not any more a bivariate $P$-polynomial association scheme. 
There still exists a family of polynomials $v_{ij}$ for which $A_{ij} = v_{ij}(A_{10}, A_{01})$, but their recurrence relations and those of the adjacency matrices lose their correspondence. A similar situation arises for association schemes based on isotropic spaces with $d > N/2$.  
 \end{rem} 

\section{Bivariate $Q$-polynomial association scheme \label{sec:Qpol}}

\subsection{Definition}

The notion of $Q$-polynomial association scheme is developed in \cite{Del} (see also \cite{BCN}) and is dual to the $P$-polynomial one.
In this section, we provide a generalization of this notion to the case of bivariate polynomials and dual to Definition \ref{def:bi}.
\begin{defi} \label{def:biQ}
Let $\cD^\star \subset \NN^2$, $0 \leq \alpha\leq 1$, $0 \leq \beta<1$ and $\preceq_{(\alpha,\beta)}$ be the order \eqref{eq:partord}.
 The association scheme with idempotents $E_0,E_1,\dots E_N$ is called bivariate $Q$-polynomial of type $(\alpha,\beta)$ on the domain $\cD^\star$ if these two conditions are satisfied:
\begin{itemize}
  \item[(i)]there exists a relabeling of the idempotents:
 \begin{equation}
  \{E_0,E_1,\dots E_N\} = \{ E_{mn} \ |\ (m,n) \in \cD^\star \}\,,
 \end{equation}
 such that, for $(m,n) \in \cD^\star$,
\begin{equation}
 \bv\, E_{mn}=v^\star_{mn}( \bv\,E_{10}, \bv\, E_{01})\quad \text{(under the Hadamard product)},
\end{equation}
 where  $v^\star_{mn}(x,y)$ is a $(\alpha,\beta)$-compatible bivariate polynomial of degree $(m,n)$;
\item[(ii)] $\cD^\star$ is $(\alpha,\beta)$-compatible.
\end{itemize}
\end{defi}

By analogy with Section \ref{sec:idem}, the adjacency matrices of the association scheme will here be denoted by $A_{\lambda}$ with $\lambda \in \cD$, where $\cD$ is a set of labels not required to be a subset of $\mathbb{N}^2$. Moreover, the dual eigenvalues defined by \eqref{eq:EA} are here written as $q_{ij}(\lambda)$, where $(i,j) \in \cD^\star$ with the same $\cD^\star$ as in Definition \ref{def:biQ}, and $\lambda \in \cD$. This highlights the fact that bivariate $Q$-polynomial association schemes of type $(\alpha,\beta)$ do not necessitate the existence of a bivariate $P$-polynomial structure. 

Let us recall that the idempotents $E_{ij}$ of an association scheme satisfy a relation dual to \eqref{eq:intern} given by 
\begin{equation}
 E_{ij} \circ E_{k\ell} = \frac{1}{\bv}\sum_{(m,n)\in\cD^\star} q_{ij,k\ell}^{mn} E_{mn}\,,
\end{equation}
where $\circ$ is the Hadamard product (or entrywise product). The numbers $q_{ij,k\ell}^{mn}$ are called Krein parameters.
The generalization of the notion of $Q$-polynomial given above leads to the following proposition constraining the Krein parameters.
\begin{prop}\label{eq:propQ} Let $\cZ$ be a symmetric association scheme with idempotents $E_{ij}$, for $(i,j)\in \cD^\star$. The following items are equivalent:
\begin{itemize}
 \item[(i)] $\cZ$ is a bivariate $Q$-polynomial association scheme of type $(\alpha,\beta)$ on $\cD^\star$;
  \item[(ii)] $\cD^\star$ is $(\alpha,\beta)$-compatible and the Krein parameters satisfy, for  $(i,j),(i+1,j) \in \cD^\star$,
  \begin{eqnarray}
   && q_{10,ij}^{i+ 1 , j}\neq 0,\ \  q_{10,i+1j}^{i , j}\neq 0\,, \label{eqq01} \\
   && q_{10,ij}^{mn}\neq 0\quad \left(\text{or}\ \  q_{10,mn}^{ij}\neq 0\right) \quad \Rightarrow \quad   (m,n)\preceq_{(\alpha,\beta)} (i+1,j)\,,\label{eqq03}
  \end{eqnarray}
  and, for  $(i,j),(i,j+1) \in \cD^\star$,
    \begin{eqnarray}
   && q_{01,ij}^{i, j+1}\neq 0,\ \ q_{01,ij+1}^{i , j}\neq 0\,,  \label{eqq04}\\
 && q_{01,ij}^{mn}\neq 0\quad \left(\text{or}\ \  q_{01,mn}^{ij}\neq 0\right)\quad \Rightarrow \quad   (m,n)\preceq_{(\alpha,\beta)} (i,j+1)\,; \label{eqq02}
  \end{eqnarray}
    \item[(iii)] $\cD^\star$ is $(\alpha,\beta)$-compatible and the dual eigenvalues $q_{ij}(\lambda) $ defined by \eqref{eq:EA} satisfy
 \begin{equation}
  q_{ij}(\lambda)=v^\star_{ij}( \theta^\star_{\lambda},  \mu^\star_{\lambda})\,, \label{eq:pv2}
 \end{equation}
 where $\theta^\star_{\lambda}=q_{10}(\lambda)$ and $\mu^\star_{\lambda}=q_{01}(\lambda)$, and $v^\star_{ij}(x,y)$ is a $(\alpha,\beta)$-compatible bivariate polynomial  of degree $(i,j)$.
\end{itemize}
\end{prop}
\proof The proof follows the same lines as the proofs of Propositions \ref{pr:z1} and \ref{pr:z2}. \endproof
This proposition leads to a generalization of the cometric property. 
An association scheme is called $(\alpha,\beta)$-cometric inside the domain $\cD^\star$ if $\cD^\star$ is $(\alpha,\beta)$-compatible and if the Krein parameters $q_{ij,k\ell}^{mn}$ satisfy the constraints \eqref{eqq01}-\eqref{eqq02}.

From now on, let $\cZ=\{ A_{ij} \ | \ (i,j)\in \cD \}$ be bivariate $P$-and $Q$-polynomial, with $\cD, \cD^* \subset \mathbb{N}^2$.
As usual, one defines renormalized polynomials as follows 
\begin{equation}
 u_{ij}(x,y)=\frac{v_{ij}(x,y)}{k_{ij}}, \qquad u^\star_{ij}(x,y)=\frac{v^\star_{ij}(x,y)}{m_{ij}},
\end{equation}
where 
\begin{equation}
 k_{ij}=p_{ij}(00), \qquad m_{mn}=q_{mn}(00)\,, \label{eq:km}
\end{equation}
 are the valency and the multiplicity, respectively. Now recall that the eigenvalues $p_{ij}(mn) $ and dual eigenvalues $q_{mn}(ij)$ of any symmetric association scheme are related as follows (see for example \cite[Theorem 3.5]{BI}),
\begin{equation}
 \frac{q_{mn}(ij) }{m_{mn}}=\frac{p_{ij}(mn) }{k_{ij}}\,. \label{eq:qp}
\end{equation}
Relation \eqref{eq:qp} implies the Wilson symmetry of the polynomials $u_{ij}$ and $u_{ij}^*$ (see \cite{Del} for the monovariate case), \textit{i.e.} one gets for $(m,n)\in \cD$ and $(i,j)\in \cD^\star$
\begin{equation}
 u_{mn}(\theta_{ij},  \mu_{ij}) =u^\star_{ij}(\theta^\star_{mn},  \mu^\star_{mn}). \label{eq:Wilson}
\end{equation}
Therefore, the polynomials $u_{mn}(x,y)$ satisfy recurrence relations but also difference equations (obtained from the recurrence relations of $u^\star_{ij}$).
They are solutions of a bispectral problem. Moreover, both polynomials $u_{ij}$ and $u^\star_{ij}$ satisfy an orthogonality relation:
\begin{equation}
	\sum_{(r,s) \in \cD^\star} m_{rs} u_{ij}(\theta_{rs},\mu_{rs})u_{mn}(\theta_{rs},\mu_{rs})=\frac{\bv}{k_{ij}}\delta_{ij,mn}\,, \label{eq:orth}
\end{equation} 
and
\begin{equation}
	\sum_{(r,s) \in \cD} k_{rs} u^\star_{ij}(\theta^\star_{rs},\mu^\star_{rs})u^\star_{mn}(\theta^\star_{rs},\mu^\star_{rs})=\frac{\bv}{m_{ij}}\delta_{ij,mn}\,. \label{eq:orths}
\end{equation}
Indeed, these are obtained by combining equations \eqref{eq:AE} and \eqref{eq:EA} in two different ways, and by using \eqref{eq:pv}, \eqref{eq:pv2} and \eqref{eq:Wilson}.

Following the proof of Section \ref{sec:ex1}, it is straightforward to show that the direct product of $Q$-polynomial associations is a bivariate $Q$-polynomial association scheme.
We treat in Section \ref{sec:exQb} another example in detail.

\subsection{Subconstituent algebra  \label{sec:sa}}

As mentioned previously the matrices $A_{ij}$ form a commutative algebra, known as the Bose--Mesner algebra. 
For any association scheme it is useful to introduce a more general algebra called the subconstituent algebra (or Terwilliger algebra) \cite{ter1,ter2,ter3}.

To do that, fix $1\leq i_0 \leq \bv$ and define the diagonal matrices $A_{mn}^\star$ as follows, for $1\leq i,j \leq \bv$,
\begin{eqnarray}
  (A^\star_{mn})_{i,j }&=&\bv \delta_{i,j} (E_{mn})_{i_0,j}.
\end{eqnarray}
These matrices satisfy 
\begin{equation}
 A^\star_{ij}  A^\star_{k\ell} = \sum_{(m,n)\in\cD^\star} q_{ij,k\ell}^{mn} A^\star_{mn}.
\end{equation}
The commutative algebra realized by these $A^\star_{ij}$ is called the dual Bose--Mesner algebra.
For a bivariate $Q$-polynomial association scheme, one gets 
\begin{equation}
 A^\star_{mn}=v^\star_{mn}( A^\star_{10}, A^\star_{01}).
\end{equation}

The algebra formed by $A_{ij}$ and $A^\star_{ij}$ is called the subconstituent algebra and is usually non-commutative.
In the case of a bivariate $P$- and $Q$-polynomial association scheme, this algebra is generated by the four elements: $A_{10}$, $A_{01}$, $A^\star_{10}$ and $A^\star_{01}$.

\subsection{Example: the symmetrization of an association scheme with two classes revisited \label{sec:exQb}}

Let $L^\star_i$ be the matrices with entries $(L^\star_i)_{jk}=q_{ik}^j$. For an association scheme with 2 classes as studied in Section \ref{sec:ssP}, one gets 
\begin{eqnarray}
 L^\star_1=\begin{pmatrix}
      0& k^\star & 0\\
      1& k^\star-1-b^\star & b^\star \\
      0&c^\star&k^\star-c^\star
     \end{pmatrix}\ ,\qquad L_2^\star=\begin{pmatrix}
      0& 0 & \frac{b^\star k^\star}{c^\star}\\
      0& b^\star &  \frac{b^\star k^\star}{c^\star}-b^\star\\
     1&k^\star-c^\star&\frac{b^\star k^\star}{c^\star}-1-k^\star+c^\star
     \end{pmatrix}\,,
\end{eqnarray}
with $k^\star= \frac{k(\tau+1)(k-\tau)}{(\tau\theta+k)(\tau-\theta)}\,,\ $
 $b^\star= -\frac{\tau(\theta+1)(k-\theta)^2}{(\tau-\theta)^2(\tau\theta+k)}\,,\ $ and $c^\star= \frac{\tau(\tau+1)(k-\tau)(k-\theta)}{(\tau-\theta)^2(\tau\theta+k)}$.

In this case, we can naturally parametrize the pairs of eigenvalues for $A_{10},A_{01}$ by the same subset $\cD^\star=\cD=\{(i,j)\ | \ i+j\leq N\}$. 
Indeed, we recall that the elements $E_{ij}$ are defined by the following formula:
\begin{equation}
	(E_0+X_1E_1+X_2E_2)^{\otimes N}=\sum_{i+j\leq N} X_1^i X_2^j E_{ij}\ .
\end{equation} 
It is immediate that $E_{ij}$ are the idempotents associated to the common eigenspaces of the matrices $A_{ij}$. 
From the formulas defining the idempotents $E_{ij}$, we have at once for the dual matrices $A^\star_{ij}$ that:
\begin{equation}
	(A^\star_0+X_1A^\star_1+X_2A^\star_2)^{\otimes N}=\sum_{i+j\leq N} X_1^i X_2^j A^\star_{ij}\ .
\end{equation} 
In other words, the matrices $A^\star_{ij}$ are obtained by applying the symmetrization process to the matrices 
$A^\star_0,A^\star_1,A^\star_2$. The recurrence relations and thus the $Q$-polynomiality of the scheme, follow exactly as in the previous section 
(one simply has to replace the parameters $k,b,c$ by the dual ones $k^\star,b^\star,c^\star$).

One thus finds that $A^\star_{ij}=v^\star_{ij}(A^\star_{10},A^\star_{01})$ where the dual polynomials $v^\star_{ij}$ satisfy the following recurrence relations:
\begin{eqnarray}
xv^\star_{ij} & = &k^\star(N-i-j+1)v^\star_{i-1,j}+\bigl(i(k^\star-1-b^\star)+j(k^\star-c^\star)\bigr)v^\star_{ij}+(i+1)v^\star_{i+1,j}\notag\\
 && +c^\star(j+1)v^\star_{i-1,j+1}+b^\star(i+1)v^\star_{i+1,j-1}\,\label{eq:recs1star},\\[0.5em]
 yv^\star_{ij} & = & \displaystyle \frac{b^\star k^\star}{c^\star}(N-i-j+1)v^\star_{i,j-1}+\bigl(b^\star i+j(\frac{b^\star k^\star}{c^\star}-1-k^\star+c^\star)\bigr)v^\star_{ij}+(j+1)v^\star_{i,j+1}\notag\\
 && \displaystyle +(j+1)(k^\star-c^\star)v^\star_{i-1,j+1}+(i+1)(\frac{b^\star k^\star}{c^\star}-b^\star)v^\star_{i+1,j-1}\,.\label{eq:recs2star}
\end{eqnarray}

Recall that the eigenvalues $(\theta_{ij},\mu_{ij})$ of $A_{10},A_{01}$, are given by
\begin{equation}
 \theta_{ij}=(N-i-j)k+i\theta+j\tau, \qquad \mu_{ij}=(N-i-j)\frac{kb}{c}-i(\theta+1)-j(\tau+1)\,.
\end{equation}
Similarly those of $A^\star_{10},A^\star_{01}$ read
\begin{equation}
 \theta^\star_{ij}=(N-i-j)k^\star+i\theta^\star+j\tau^\star, \qquad \mu_{ij}=(N-i-j)\frac{k^\star b^\star}{c^\star}-i(\theta^\star+1)-j(\tau^\star+1)\,.
\end{equation}
Let us also remark that the relation 
\begin{eqnarray}
 J(x) E_{00}= [(A_0+x_1A_1+x_2A_2)E_0]^{\otimes N} =\sum_{i+j\leq N} \frac{N!}{i!j!(N-i-j)!}(kx_1)^i\left(\frac{bk}{c}x_2\right)^j E_{00}\,,
\end{eqnarray}
implies that 
\begin{equation}
 p_{ij}(0,0)= \frac{N!}{i!j!(N-i-j)!}k^i\left(\frac{bk}{c}\right)^j\,.
\end{equation}
Similarly $q_{ij}(0,0)$ is obtained by replacing $k,\tau,\theta$ by $k^\star,\tau^\star,\theta^\star$ in the previous relation.
The Wilson symmetry reads in this case
\begin{equation}
 \frac{i!j!(N-i-j)!}{k^i (bk/c)^j} v_{ij}(\theta_{mn},\mu_{mn}) =\frac{m!n!(N-m-n)!}{(k^\star)^m (b^\star k^\star/c^\star)^n} v^\star_{mn}(\theta^\star_{ij},\mu^\star_{ij})\,.
\end{equation}
This makes explicit that this scheme obtained by symmetrization is bivariate $P$- and $Q$-polynomial.

\section{Outlook}\label{sec:outlooks}

All the notions introduced in this paper should be generalizable to more than two variables. This would lead to the definition of a multivariate $P$-polynomial association scheme but 
would require a generalization of the definition \ref{def:compa} about the compatibility of a multivariate polynomial.
The direct product of $n$ different $P$-polynomial association schemes or the symmetrization of a $P$-polynomial association scheme with $n$ classes would be examples of $n$-variate 
$P$-polynomial association schemes. It would also be interesting to find more examples of association schemes with the bivariate (or more generally, multivariate) $Q$-polynomial property. 
A first step could be to investigate if this property is present in all the examples presented in Section \ref{sec:ex}. For the non-binary Johnson scheme, the dual eigenvalues are given in \cite{Dun,TAG}, and 
for the association schemes based on attenuated spaces, in \cite{Kur}. These results may allow us to prove that these schemes are bivariate $Q$-polynomial using Propoition \ref{eq:propQ}.

In the monovariate case, it is well-known that $P$-polynomial association schemes correspond to distance-regular graphs \cite{BI}. A
bivariate generalization of the concept of distance-regular graph that would be equivalent to the notion of bivariate $P$-polynomial association scheme could be explored.    

The subconstituent algebra of a $P$- and $Q$-polynomial association scheme has been characterized. Indeed, in this case, the generators of the subconstituent 
algebra are $(A_1,A_1^\star)$ and they satisfy the tridiagonal relations \cite{ter3}. This has also led to the definition and study of Leonard pairs and their relation with 
the Askey--Wilson relations \cite{TV}. It would be quite interesting to extend this for bivariate $P$- and $Q$-polynomial association schemes. Indeed, as explained in Section \ref{sec:sa},
the subconstituent algebra is generated by four elements and it would be enlightening to find the relations between them. Those relations might depend on the parameters $\alpha$ and $\beta$
defining the type of bivariate association scheme.  Such an algebraic approach holds hopes of a classification of the 
bivariate $P$- and $Q$-polynomial association schemes. 

The generalization of Leonard pairs would
also be natural in this context (see the conclusion of \cite{IT} where a definition is already proposed). One could then speak of higher rank Leonard pairs. 
The classification of these pairs would be very interesting in the context of the bispectral multivariate polynomials.
 
\vspace{1cm}

\noindent \textbf{Acknowledgements.} The authors are grateful to Xiaohong Zhang for her comments, corrections and suggestions. They are also thankful to professeor Eiichi Bannai for his corrections. 
PAB and MZ hold an Alexander-Graham-Bell scholarship from the Natural Sciences and Engineering Research Council of Canada (NSERC).
NC and LPA thank the CRM for its hospitality and are supported by the international research project AAPT of the CNRS and the ANR Project AHA ANR-18-CE40-0001. 
The research of LV is supported by a Discovery Grant from the Natural Sciences and Engineering Research Council (NSERC) of Canada.

\end{document}